\documentclass[11pt]{amsart}
\usepackage{amsmath,amsthm,amsfonts,amssymb,amstext}
\theoremstyle{plain}

\usepackage{amssymb}
\usepackage{amscd}

\newtheorem{lemma}{Lemma}[section]

\newtheorem{prop}{Proposition}[section]

\theoremstyle{remark}

\def\B{\operatorname{B}}
\def\C{\operatorname{C}}
\def\D{\operatorname{D}}
\def\E{\operatorname{E}}
\def\F{\operatorname{F}}
\def\G{\operatorname{G}}

\def\I{\operatorname{I}}
\def\J{\operatorname{J}}

\def\O{\operatorname{O}}
\def\U{\operatorname{U}}

\def\SL{\operatorname{SL}}
\def\GL{\operatorname{GL}}

\def\SO{\operatorname{SO}}
\def\Sp{\operatorname{Sp}}
\def\SU{\operatorname{SU}}

\def\PSO{\operatorname{PSO}}

\def\PSU{\operatorname{PSU}}

\def\Ad{\operatorname{Ad}}

\def\Aut{\operatorname{Aut}}

\def\diag{\operatorname{diag}}

\def\det{\operatorname{det}}
\def\dim{\operatorname{dim}}

\def\exp{\operatorname{exp}}

\def\Hom{\operatorname{Hom}}
\def\id{\operatorname{id}}

\def\Im{\operatorname{Im}}

\def\Int{\operatorname{Int}}

\def\Lie{\operatorname{Lie}}

\def\Pin{\operatorname{Pin}}
\def\rank{\operatorname{rank}}

\def\span{\operatorname{span}}
\def\Spin{\operatorname{Spin}}
\def\Stab{\operatorname{Stab}}

\newcommand{\fra}{\mathfrak{a}}

\newcommand{\frf}{\mathfrak{f}}
\newcommand{\frg}{\mathfrak{g}}
\newcommand{\frh}{\mathfrak{h}}

\newcommand{\frl}{\mathfrak{l}}

\newcommand{\frs}{\mathfrak{s}}
\newcommand{\frt}{\mathfrak{t}}
\newcommand{\fru}{\mathfrak{u}}

\newcommand{\bbC}{\mathbb{C}}

\newcommand{\bbR}{\mathbb{R}}
\newcommand{\bbZ}{\mathbb{Z}}

\begin{document}

\title[Maximal abelian subgroups]{Maximal abelian subgroups of Spin groups and some exceptional simple
Lie groups}
\author{Jun Yu}
\date{}

\abstract{We classify closed abelian subgroups of the simple groups $\G_2$, $\F_4$, $\Aut(\mathfrak{so}(8))$
having centralizer the same dimension as the dimension of the subgroup, as well as finite abelian subgroups
of certain spin and half-spin groups having finite centralizer.}
\endabstract

\maketitle

%\date{August 2011}


\tableofcontents

\section{Introduction}

In this paper we continue the classification of maximal abelian subgroups of compact simple Lie groups in
\cite{Yu2}, where we give a classification of such subgroups of matrix groups. The same as in \cite{Yu2}, we
study closed abelian subgroups $F$ of a compact simple Lie group $G$ satisfying the condition of
\[\dim\frg_0^{F}=\dim F,\tag{$\ast$}\] where $\frg_0$ is the Lie algebra of $G$. Precisely, we discuss such
subgroups of compact simple Lie groups $\G_2$, $\F_4$, and of the automorphism group $\Aut(\mathfrak{so}(8))$
of $\mathfrak{so}(8)$, as well as finite abelian subgroups of the spin group $\Spin(n)$ for $n=7$, $8$, $10$,
$12$, $14$, and of the half-spin group $\Spin(12)/\langle c\rangle$.

Given a finite abelian subgroup $F$ of $\Spin(n)$ satisfying the condition $(\ast)$, one can show that the
projection of $F$ to $\SO(n)$ is diagonalizable. Hence we could assume that $$F\subset K_{n}:=
\langle e_1e_2,e_1e_3,\dots,e_1e_{n}\rangle.$$ One has $\Spin(n)^{e_1e_2\dots e_{k}}\cong(\Spin(k)\times
\Spin(n-k))/\langle(-1,-1)\rangle$. In low ranks, we have $\Spin(3)\cong\Sp(1)$, $\Spin(4)\cong\Sp(1)\times
\Sp(1)$, $\Spin(5)\cong\Sp(2)$, $\Spin(6)\cong\SU(4)$, which are matrix groups. For an even integer $n\leq 14$,
in the case of $F$ contains elements in specific conjugacy classes, we are able to
classify $F$ satisfying the condition $(\ast)$ by reducing to matrix groups. By a combinatorial argument,
we are also able to show that $F$ always contains elements in these specific conjugacy classes if it satisfies
the condition $(\ast)$. The group $\Spin(12)/\langle c\rangle$ is closely related to the group
$\SO(12)/\langle-I\rangle$, abelian subgroups of the latter are well understood in \cite{Yu2}. To classify
closed abelian subgroups of $\G_2$ or $\F_4$ satisfying the condition $(\ast)$, we need to employ a new method.
By Lemma \ref{L:center}, we reduce this classification to the classification of finite abelian subgroups of
derived groups of Levi subgroups, which are matrix groups or the group itself. The difficult part is the
classification of finite abelian subgroups of $\G_2$ or $\F_4$ itself. For this we recall a theorem of Steinberg
which indicates that the fixed point subgroup $G^{\sigma}$ is connected if $G$ is a connected and simply
connected compact Lie group and $\sigma$ is an automorphism of it. For $G=\G_2$ or $\F_4$, since it is
connected and simply connected, the centralizer $G^{x}$ must be semisimple if $x$ is an element of a finite
abelian subgroup $F$ of $G$ satisfying the condition $(\ast)$. Full rank semisimple subgroups are well
understood (cf. \cite{Oshima}). Therefore we could identify the possible conjugacy class of the subgroup $
G^{x}$ and also the conjugacy class of $x$. Again $G^{x}$ is a matrix group if $x$ is a non-identity element,
by which we can use results from \cite{Yu2}. The group $\Aut(\mathfrak{so}(8))$ looks like both a matrix group
and an exceptional simple Lie group. Its component group is isomorphic to $S_3$. Given an abelian subgroup $F$
of $G=\Aut(\mathfrak{so}(8))$, in the case of $F$ does not meet the connected components with outer automorphisms
of order 3, $F$ is conjugate to a subgroup of $\O(8)/\langle-I\rangle$, which is studied in \cite{Yu2}. In the
case of $F$ meets such connected components, Lemma \ref{L:center-outer} reduces the classification to the
classification of finite abelian subgroups of derived groups of (disconnected) Levi subgroups. Again the
difficult part is for finite abelian subgroups of $G$ itself. For this we use Steinberg's theorem (applied to
outer automorphisms) and reduce the classification to finite abelian subgroups of $\G_2$ or $\PSU(3)$. We also
describe Weyl groups of maximal abelian subgroups. This is achieved through the description of fusions of maximal
abelian subgroups in the classification. Recall that the fusion of an abelian subgroup $F$ is the distribution
of conjugacy classes of elements of $F$.

Note that maximal abelian subgroups of $\G_2$, $\F_4$, $\Aut(\mathfrak{so}(8))$ are in one-to-one
correspondence with fine group gradings of the complex simple Lie algebras $\frg_{2}(\bbC)$, $\frf_4(\bbC)$,
$\mathfrak{so}(8,\bbC)$. These fine group gradings are known in the literature (cf. \cite{Elduque-Kochetov}).
Our approach through the study of maximal abelian subgroups not only gives this classifiaction in an easy way,
but also describes the Weyl groups very well. Finally, we remark that for the study of closed abelian subgroups
of simple groups of type $\bf E_6$, $\bf E_7$ or $\bf E_8$, we will employ the same method as the method for
$\bf G_2$, $\bf F_4$ and $\bf D_4$. On the other hand, the finite abelian subgroups of spin groups and half-spin
groups obtained here will appear in the study of non-finite closed abelian subgroups of $\E_6$, $\E_7$ and
$\E_8$ satisfying the condition $(\ast)$.

\smallskip

\noindent{\it Notation and conventions.} Given a Euclidean linear space $V$ of dimension $n$ with an
orthonormal basis $\{e_1,e_2,\dots,e_{n}\}$, denote by $\Pin(n)$ ($\Spin(n)$) the pin (spin) group of
degree $n$ associated to $V$. Write \[c=c_{n}:=e_1e_2\cdots e_{n}\in\Pin(n).\] Then, $c\in\Spin(n)$ if
and only if $n$ is even. In this case $c$ lies in the center of $\Spin(n)$. Let $I_{n}$ be the $n\times n$
identity matrix. We define the following matrices,
\[I_{p,q}=\left(\begin{array}{cc} -I_{p}&0\\ 0&I_{q}\\\end{array}\right),\
J_{n}=\left(\begin{array}{cc} 0&I_{n}\\ -I_{n}&0\\\end{array} \right),\
\J'_{n}=\left(\begin{array}{cc}0&\I_{n}\\\I_{n}&0\\\end{array}\right).\]
%K_{n}=\left(\begin{array}{cccc}0&\I_{n}&0&0\\-\I_{n}&0&0&0\\0&0&0&-\I_{n}\\
%0&0&\I_{n}&0\\\end{array}\right).
%\[K_{n}=\left(\begin{array}{cccc} 0&0&0&I_{n}\\0&0&-I_{n}&0\\ 0&I_{n}&0&0\\-I_{n}&0&0&0\\\end{array}\right).\]
Given a compact Lie group $G$ and a closed abelian subgroup $H$, let $W(H)=N_{G}(H)/C_{G}(H)$ and we call
it the Weyl group of $H$ (in $G$). Denote by $\mathbb{H}$ a quaternion algebra over $\mathbb{R}$ with
standard basis $\{1,\textbf{i},\textbf{j},\textbf{k}\}$. The symplectic group
$\Sp(1)=\{a+b\textbf{i}+c\textbf{j}+d\textbf{k}: a,b,c,d\in\mathbb{R},a^2+b^2+c^2+d^2=1\}$. Note that
we have $\Sp(1)\cong\SU(2)$. In this paper we always use $\Sp(1)$ instead of $\SU(2)$ as it is
easier to write a quaternion number instead of a $2\times 2$ matrix.

%Complex simple Lie algebra $\frg$ and a compact real form $\fru_0$, \cite{Yu}.

\smallskip

\noindent{\it Acknowledgement.} We thank the anonymous referee for careful reading and helpful
comments.

%Given a compact semisimple Lie algebra $\fru_0$, let $\Aut(\fru_0)$ be the group of automorphisms
%of $\fru_0$ and $\Int(\fru_0)=\Aut(\fru_0)_0$ be the group of inner automorphisms.

\section{Spin and half-spin groups}

Let $G=\Spin(n)$ and $\pi:\Spin(n)\longrightarrow\SO(n)$ be the natural projection. Given a finite
abelian subgroup $F$ of $G$ satisfying the condition $(*)$, without loss of generality we assume
that $-1\in F$ and furthermore $c\in F$ if $2|n$. Since $F$ satisfies the condition $(\ast)$, as
a finite abelian subgroup of $\SO(n)$, $\pi(F)$ also satisfies the condition $(\ast)$. By this, the
centralizer $\SO(n)^{\pi(x)}$ has finite center for any $x\in F$. By calculation one shows that
$\pi(x)\sim I_{p,n-p}$ for some $0\leq p\leq n$. Then, an inductive process shows that $\pi(F)$ is
a diagonalizable of $\SO(n)$. By this we may and do assume that $\pi(F)$ is contained in the subgroup
of diagonal matrices in $\SO(n)$. Hence $F\subset\langle e_1e_2,e_1e_3,\dots,e_1e_{n}\rangle$. For
an index set $I=\{i_1,i_2,\dots,i_{k}\}\subset\{1,2,\dots,n\}$, $i_1<i_2<\cdots<i_{k}$, let
\[e_{I}=e_{i_1}e_{i_2}\cdots e_{i_{k}}.\] For an even integer $k$, let
\[X_{k}=\{I\subset\{1,2,\dots,n\}: e_{I}\in F\textrm{ and }|I|=k\}.\] For any $i,j$, $1\leq i<j\leq n$,
since \[\Spin(n)^{e_{i}e_{j}}\cong(\Spin(n-2)\times\Spin(2))/\langle(-1,-1)\rangle\] has center of
positive dimension, $F$ does not contain the element $e_{i}e_{j}$. Thus $X_2=\emptyset$.

\begin{lemma}\label{L:Spin-1}
For $n\geq 3$, let $F\subset\Spin(n)$ be a finite abelian subgroup satisfying the condition $(*)$. Then
$F/(F\cap\langle-1,c\rangle)$ is an elementary abelian 2-group and
\[\log_{2}n\leq\rank(F/(F\cap\langle-1,c\rangle))\leq\frac{n-1}{2}.\]
\end{lemma}
\begin{proof}
By the above remark, we assume that $F\subset K_{n}=\langle e_1e_2,e_1e_3,\dots,e_1e_{n}\rangle$.
Thus $F/(F\cap\langle-1,c\rangle)$ is an elementary abelian 2-group. Let
$r=\rank(F/(F\cap\langle-1,c\rangle))$. Then, $\O(n)^{\pi(F)}$ has a blockwise decomposition of at most
$2^{r}$ components. Hence $\mathfrak{so}(n)^{F}=0$ implies that $2^{r}\geq n$. Thus $r\geq\log_{2}n$.
For $x,y\in K_{n}$, define $m(x,y)=xyx^{-1}y^{-1}$. This gives a bimultiplicative function
\[m: K_{n}\times K_{n}\longrightarrow\{\pm{1}\}.\] It is clear that $\ker m=\Spin(n)\cap\langle-1,c\rangle$
and $\rank(K_{n}/K_{n}\cap\langle-1,c\rangle)=2[\frac{n-1}{2}]$. Since $F$ is an abelian subgroup, one
has $m|_{F\times F}=1$. Therefore
\[\rank(F/(F\cap\langle-1,c\rangle))\leq\frac{1}{2}\rank(K_{n}/\langle-1,c\rangle)=[\frac{n-1}{2}].\]
\end{proof}

\begin{prop}\label{P:Spin-2}
For $2\leq n\leq 6$, there exist no finite abelian subgroups $F$ of $\Spin(n)$ satisfying the
condition $(*)$.
\end{prop}

\begin{proof}
Let $n\geq 3$ and $F\subset\Spin(n)$ be a finite abelian subgroup satisfying the condition $(*)$. Let
$r=\rank(F/(F\cap\langle-1,c\rangle))$. By Lemma \ref{L:Spin-1}, $\log_{2}n\leq r\leq\frac{n-1}{2}$. Hence
$2^{[\frac{n-1}{2}]}\geq n$. If $n=2m$ is even, then $2^{m-1}\geq 2m$ and hence $m\geq 4$, $n\geq 8$.
If $n=2m-1$ is odd, then $2^{m-1}\geq 2m-1$ and hence $m\geq 4$, $n\geq 7$.
\end{proof}

Let $$F_8=\langle-1,c,e_1e_2e_3e_4,e_1e_2e_5e_6,e_1e_3e_5e_7\rangle,$$ $$F_7=\langle-1,e_1e_2e_3e_4,
e_1e_2e_5e_6,e_1e_3e_5e_7\rangle.$$ Then, $F_8$ is a maximal finite abelian subgroup of $\Spin(8)$ and
$F_7$ is a maximal finite abelian subgroup of $\Spin(7)$.

\begin{prop}\label{P:Spin-3}
Any finite abelian subgroup $F$ of $\Spin(8)$ satisfying the condition $(*)$ and with $-1,c\in F$ is
conjugate to $F_8$. Any finite abelian subgroup $F$ of $\Spin(7)$ satisfying the condition $(*)$ and
with $-1\in F$ is conjugate to $F_7$. There exist no finite abelian subgroups of $\Spin(10)$ satisfying
the condition $(*)$.
\end{prop}

\begin{proof}
For $n=8$, without loss of generality we assume that $$F\subset\langle e_1e_2,e_1e_3,\dots,e_1e_8\rangle.$$
Since $e_{i}e_{j}\not\in F$ for any $i\neq j$, besides $\pm{1},\pm{c}$, any element of $F$ is of the form
$\pm{}e_{I}$ for some $I\subset\{1,2,\dots,8\}$ with $|I|=4$. We may and do assume that $e_1e_2e_3e_4\in F$.
Then \[F\subset\Spin(8)^{e_1e_2e_3e_4}\cong(\Sp(1)\times\Sp(1)\times\Sp(1)\times\Sp(1))/\langle(-1,-1,-1,-1)
\rangle.\] Here we use $\Spin(4)\cong\Sp(1)\times\Sp(1)$. Any finite abelian subgroup of $\Sp(1)/\langle-1
\rangle$ satisfying the condition $(*)$ is conjugate to $\langle[\textbf{i}],[\textbf{j}]\rangle$. Hence
\[F\sim\langle[(-1,1,1,1)],[(1,-1,1,1)],[(1,1,-1,1)],[(\textbf{i},\textbf{i},\textbf{i},\textbf{i})],
[(\textbf{j},\textbf{j},\textbf{j},\textbf{j})]\rangle.\] This shows the uniqueness of the conjugacy class
when $n=8$. The proof for the $n=7$ case is similar as that for the $n=8$ case. For $n=10$, without loss of
generality we assume that $F\subset\langle e_1e_2,e_1e_3,\dots,e_1e_{10}\rangle$ and $c\in F$. Since
$e_{i}e_{j}\notin F$ for any $i\neq j$, besides $\pm{1},\pm{c}$, any element of $F$ is of the form $\pm{}e_{I}$
with $|I|=4$ or $6$. We may and do assume that $e_1e_2e_3e_4\in F$. Then \[F\subset\Spin(10)^{e_1e_2e_3e_4}
\cong(\Sp(1)\times\Sp(1)\times\SU(4))/\langle(-1,-1,-I)\rangle.\] Here we use $\Spin(4)\cong\Sp(1)\times\Sp(1)$
and $\Spin(6)\cong\SU(4)$. Let $p_1$, $p_2$, $p_3$ be the projections of $F$ to three components
$\Sp(1)/\langle-1\rangle$, $\Sp(1)/\langle-1\rangle$, $\SU(4)/\langle-I\rangle$ and $F'_1$, $F'_2$, $F'_3$ be
the images, which are elementary abelian 2-subgroups by \cite{Yu2}, Proposition 2.1. We have bimultiplicative
functions $m_1$, $m_2$, $m_3$ on $F'_1$, $F'_2$, $F'_3$ taking values in $\{\pm{1}\}$. By the definition of
$m_{i}$ in \cite{Yu2}, we have  $m_{i}(p_{i}(x),p_{i}(y))=m_{j}(p_{j}(x),p_{j}(y))$ for any $x,y\in F$ and
$i,j\in\{1,2,3\}$. Hence $\rank(F'_{i}/\ker m_{i})=\rank(F'_{j}/\ker m_{j})$. It follows from \cite{Yu2},
Proposition 2.1 that $\ker m_1=\ker m_3=1$, $F'_1\cong(C_2)^{2}$ and $F'_3\cong (C_2)^{4}$. Hence
$\rank(F'_1/\ker m_1)\neq\rank(F'_3/\ker m_3)$, which is a contradiction. Thus $\Spin(10)$ has no finite
abelian subgroups satisfying the condition $(*)$.
\end{proof}

%From the projection \[\Sp(1)\times\Sp(1)\times\SU(4)\longrightarrow(\Sp(1)\times\Sp(1)\times
%\SU(4))/\langle(-1,-1,-I)\rangle,\]

Since every element in $F_{7}-\{1\}$ is conjugate to $e_1e_2e_3e_4$, one can show that $W(F_7)\cong
\mathbb{F}_2^{3}\rtimes\GL(3,\mathbb{F}_2)$. Similarly every element in $F_{8}-\langle-1,c\rangle$ is
conjugate to $e_1e_2e_3e_4$ and hence $W(F_8)\cong\Hom(\mathbb{F}_2^{3},\mathbb{F}_2^{2})\rtimes
\GL(3,\mathbb{F}_2)$.

\smallskip

Let \[F_{12}=\langle -1,c,e_1e_2e_3e_4,e_1e_2e_5e_6,e_1e_2e_7e_8,e_1e_2e_{9}e_{10},e_1e_3e_5e_7e_9e_{11}
\rangle,\] \[F_{14}=\langle-1,c,e_1e_2e_3e_4,e_1e_2e_5e_6,e_1e_3e_5e_7,e_8e_9e_{10}e_{11},
e_8e_9e_{12}e_{13},e_8e_{10}e_{12}e_{14}\rangle.\] Then, $F_{12}$ is a maximal finite abelian subgroup of
$\Spin(12)$ and $F_{14}$ is a maximal finite abelian subgroup of $\Spin(14)$.

\begin{prop}\label{P:Spin-4}
Any finite abelian subgroup $F$ of $\Spin(12)$ satisfying the condition $(*)$ and with $-1,c\in F$ is
conjugate to $F_{12}$. Any finite abelian subgroup $F$ of $\Spin(14)$ satisfying the condition $(*)$ and
with $-1,c\in F$ is conjugate to $F_{14}$.
\end{prop}

\begin{proof}
For $n=12$, without loss of generality we assume that $$F\subset\langle e_1e_2,e_1e_3,\dots,e_1e_{12}
\rangle.$$ Since $e_{i}e_{j}\notin F$ for any $i\neq j$, besides $\pm{1},\pm{c}$, any element of $F$ is of
the form $\pm{}e_{I}$ with $|I|=4$, $6$ or $8$. If $F$ contains an element of the form $e_{I}$ with $|I|=6$,
we may and do assume that $e_1e_2e_3e_4e_5e_6\in F$. Then \[F\subset\Spin(12)^{e_1e_2e_3e_4e_5e_6}\cong
(\SU(4)\times\SU(4))/\langle(-I,-I)\rangle\] and $[(iI,I)],[(I,iI)]\in F$. Here we use $\Spin(6)\cong\SU(4)$.
By \cite{Yu2}, Proposition 2.1, the image of the projection of $F$ to each component $\SU(4)/\langle-I\rangle$
is conjugate to $\langle [iI],[E_1],[E_2],[E_3],[E_4]\rangle$, where $E_1=I_{2,2}$, $E_2=J'_2$,
$E_3=\diag\{I_{1,1},I_{1,1}\}$, $E_4=\diag\{J'_{1},J'_{1}\}$. Hence \[F\sim\langle[(iI,I)],[(I,iI)],
[(E_1,E_1)],[(E_2,E_2)],[(E_3,E_3)],[(E_4,E_4)]\rangle.\] Therefore the conjugacy class of $F$ is unique in
this case. Moreover $F_{12}$ is in this conjugacy class since it contains an element of the form $e_{I}$ with
$|I|=6$. If $F$ contains two elements of the form $e_{I_1},e_{I_2}$ with $|I_1|=|I_2|=4$ and $I_1\cap I_2=
\emptyset$, we may and do assume that $e_1e_2e_3e_4,e_5e_6e_7e_8\in F$. Then, \begin{eqnarray*}&&F\\&\subset&
\Spin(12)^{e_1e_2e_3e_4,e_5e_6e_7e_8}\\&\cong&(\Sp(1)^{2}\times\Sp(1)^{2}\times\Sp(1)^{2})/\langle(-1,-1,-1,
-1,1,1),(1,1,-1,-1,-1,-1)\rangle\end{eqnarray*} and $[(\pm{1},\pm{1},\pm{1},\pm{1},\pm{1},\pm{1})]\in F$.
Here we use $\Spin(4)\cong\Sp(1)\times\Sp(1)$. One can show that the image of the projection of $F$ to each
component $\Sp(1)^{2}/\langle(-1,-1)\rangle$ is conjugate to $\langle[(-1,1)],[(\textbf{i},\textbf{i})],
[(\textbf{j},\textbf{j})]\rangle$. Hence \begin{eqnarray*}F\sim\langle[(\pm{1},\pm{1},\pm{1},\pm{1},\pm{1},
\pm{1})],[(\textbf{i},\textbf{i},\textbf{i},\textbf{i},1,1)],[(\textbf{j},\textbf{j},\textbf{j},\textbf{j},
\textbf{j},\textbf{j})],[(1,1,\textbf{i},\textbf{i},\textbf{i},\textbf{i})]\rangle.\end{eqnarray*} Therefore
the conjugacy class of $F$ is unique in this case. Moreover $F_{12}$ is in this conjugacy class since it
contains two elements of the form $e_{I_1},e_{I_2}$ with $|I_1|=|I_2|=4$ and $I_1\cap I_2=\emptyset$.
Suppose the above two cases do not happen. Then any element of $F$ besides $\pm{1},\pm{c}$ is of the form
$\pm{}e_{I}$ with $|I|=4$ or $8$; and for any two $e_{I_1},e_{I_2}\in F$ ($I_1\neq I_2$) with $|I_1|=|I_2|=4$,
one has $|I_1\cap I_2|=2$. By Lemma \ref{L:Spin-1}, $\rank(F/\langle-1,c\rangle)\geq 4$. Without loss of
generality we assume that $e_1e_2e_3e_4,e_1e_2e_5e_6,e_1e_3e_5e_7\in F$. One can show that there exist no
subset $I$ of $\{1,2,\dots,12\}$ with cardinality 4 and the intersection of $I$ with each of $\{1,2,3,4\}$,
$\{1,2,5,6\}$, $\{3,4,5,6\}$, $\{1,3,5,7\}$, $\{2,4,5,7\}$, $\{2,3,6,7\}$, $\{1,4,6,7\}$ has cardinality 2.

For $n=14$, without loss of generality we assume that $$F\subset\langle e_1e_2,e_1e_3,\dots,e_1e_{14}
\rangle.$$ Since $e_{i}e_{j}\notin F$ for any $i\neq j$, besides $\pm{1},\pm{c}$, any element of $F$ is of
the form $\pm{}e_{I}$ with $|I|=4$, $6$, $8$ or $10$. If $F$ contains two elements of the form $e_{I_1},
e_{I_2}$ with $|I_1|=|I_2|=4$ and $I_1\cap I_2=\emptyset$, we may and do assume that $e_1e_2e_3e_4,
e_5e_6e_7e_8\in F$. Then \begin{eqnarray*}&&F\\&\subset&\Spin(14)^{e_1e_2e_3e_4,e_5e_6e_7e_8}\\&\cong&
(\Sp(1)^{2}\times\Sp(1)^{2}\times\SU(4))/\langle(-1,-1,1,1,-I),(1,1,-1,-1,-I)\rangle\end{eqnarray*} and
$[(\pm{1},\pm{1},\pm{1},\pm{1},I)],[(1,1,1,1,iI)]\in F$. By \cite{Yu2}, Proposition 2.1, the image of the
projection of $F$ to each of the first two components $\Sp(1)^{2}/\langle(-1,-1)\rangle$ is conjugate to
$\langle[(-1,1)],[(\textbf{i},\textbf{i})],[(\textbf{j},\textbf{j})]\rangle$ and the image of the projection
of $F$ to the third component $\SU(4)/\langle-I\rangle$ is conjugate to $\langle[iI],[E_1],[E_2],[E_3],[E_4]
\rangle$. Then, one can show that \begin{eqnarray*}F\sim&&\langle[(\pm{1},\pm{1},\pm{1},\pm{1},I)],
[(1,1,1,1,iI)],[(\textbf{i},\textbf{i},1,1,E_1)],[(\textbf{j},\textbf{j},1,1,E_2)],\\&&
[(1,1,\textbf{i},\textbf{i},E_3)],[(1,1,\textbf{j},\textbf{j},E_4)]\rangle.\end{eqnarray*} Hence the conjugacy
class of $F$ is unique in this case. Moreover $F_{14}$ is in this conjugacy class since it contains two
elements of the form $e_{I_1},e_{I_2}$ with $|I_1|=|I_2|=4$ and $I_1\cap I_2=\emptyset$. Suppose the above
case does not happen. Then, for any $e_{I_1},e_{I_2}\in F$ ($I_1\neq I_2$) with $|I_1|=|I_2|=4$, one has
$|I_1\cap I_2|=2$. Let $F'=\{\pm{}e_{I}\in F: |I|=4\}\cup\{\pm{1}\}$. Then $F'$ is a subgroup of $F$.
Moreover, there exists a subgroup $F''$ of $F$ such that: $-1\in F''$; for any $\pm{e_{I}}\in F''$, $|I|=8$;
and \[F/\langle-1,c\rangle=(\langle F',c\rangle/\langle-1,c\rangle)\times (\langle F'',c\rangle/\langle-1,c
\rangle)\] is a direct product. Hence for any $e_{I_1}\in F',e_{I_2}\in F''$, $|I_1\cap I_2|=2$; and for
any $e_{I_1},e_{I_2}\in F''$ ($I_1\neq I_2$), $|I_1\cap I_2|=4$. Let $r_1=\rank(F'/\langle -1\rangle)$,
$r_2=\rank F''/\langle-1\rangle$ and $r=\rank(F/\langle-1,c\rangle)$. Then $r=r_1+r_2$. By Lemma
\ref{L:Spin-1}, $r\geq 4$. Since for any $e_{I_1},e_{I_2}\in F''$ ($I_1\neq I_2$), $|I_1\cap I_2|=4$, one
has $r_2\leq 3$. Since for any $e_{I_1},e_{I_2}\in F'$ ($I_1\neq I_2$), $|I_1\cap I_2|=2$, one has
$r_1\leq 3$. If $r_1=3$, we may and do assume that $F'=\langle -1,e_1e_2e_3e_4,e_1e_2e_5e_6,e_1e_3e_5e_7
\rangle$. Choose an $e_{I}\in F''$ with $|I|=8$, one can show that $|I\cap I_1|=2$ for any
$e_{I_1}\in F'$ can not hold. If $r_1=2$, we may and do assume that $F'=\langle-1,e_1e_2e_3e_4,
e_1e_2e_5e_6\rangle$. Any element of $F''$ can be chosen of the form $\pm{}e_1e_3e_{I}$ with
$I\subset\{5,6,7,8,9,10,11,12,13,14\}$ and $|I\cap\{5,6\}|=1$. Then, $r_2\leq 1$ and hence
$\mathfrak{so}(14)^{F}\neq 0$ in this case. If $r_1=1$, then $r_2=3$ as $r\geq 4$. In this case one can
show that $\mathfrak{so}(14)^{F}\neq 0$ as well.
\end{proof}

For $F_{12}$, denote by \[K=\langle -1,c,e_1e_2e_3e_4,e_1e_2e_5e_6,e_1e_2e_7e_8,
e_1e_2e_{9}e_{10}\rangle.\] Since $\mathfrak{so}(12)^{K}=\oplus_{1\leq i\leq 6}\mathfrak{so}(2)$, one has
$W(K)\cong\mathbb{F}_2^{4}\rtimes S_6$, where $\mathbb{F}_2^{4}=A_1/A_2$,
$A_1=\{(a_1,a_2,a_3,a_4,a_5,a_6)\in\mathbb{F}_2^{6}:a_1+a_2+a_3+a_4=a_5+a_6=0\}$,
$A_2=\mathbb{F}_2\cdot(1,1,1,1,1,1)\subset\mathbb{F}_2^{6}$. Here $S_6$ is generated by the conjugation
action of elements like $\frac{1+e_{2i+1}e_{2i+3}}{\sqrt{2}}\frac{1+e_{2i+2}e_{2i+4}}{\sqrt{2}}$,
$0\leq i\leq 4$, and $\mathbb{F}_2^{4}$ is generated by the conjugation action of elements like
$e_{2i+1}e_{2i+3}$, $0\leq i\leq 4$. It is clear that $K$ is invariant under $W(F_{12})$ and hence there
is a homomorphism $p: W(F_{12})\rightarrow W(K)$, which is apparently a surjective map. One can show that
$\ker p=\Hom(F_{12}/K,K)$, generated by the conjugation action of elements like $e_1e_2$,
$\frac{1+e_{2i+1}e_{2i+2}}{\sqrt{2}}\frac{1+e_{2i+3}e_{2i+4}}{\sqrt{2}}$, $0\leq i\leq 4$. Therefore we
get $$W(F_{12})\cong\mathbb{F}_{2}^{6}\rtimes(\mathbb{F}_2^{4}\rtimes S_6).$$ For $F_{14}$, denote by
$$A=\langle-1,e_1e_2e_3e_4,e_1e_2e_5e_6,e_1e_3e_5e_7\rangle,$$
$$B=\langle-1,e_8e_9e_{10}e_{11},e_8e_9e_{12}e_{13},e_8e_{10}e_{12}e_{14}\rangle.$$ Then, $\{A,B\}$ is
invariant under $W(F_{14})$. It is clear that $$W(F_{14})=(W(A)\times W(B))\rtimes S_2,$$ where
$W(A)\cong W(B)\cong\mathbb{F}_2^{3}\rtimes\GL(3,\mathbb{F}_2)$.

\smallskip

Given $n\geq 3$, let $k_{n}$ be the number of conjugacy classes of finite abelian subgroups $F$ of $\Spin(n)$
satisfying the condition $(*)$, with $-1\in F$ and $c_{n}\in F$ if $2|n$.

\begin{prop}\label{P:Spin-5}
For any $n>1$ except $n=2$, $3$, $4$, $5$, $6$ or $10$, $\Spin(n)$ has a finite abelian subgroup satisfying the
condition $(*)$.
\end{prop}

\begin{proof}
We show that $k_{2n-1}\geq k_{2n-2}+k_{2n}$ for any integer $n>1$. Finite abelian subgroups $F$ of $\Spin(2n-1)$
satisfying the condition $(*)$ and with $-1\in F$ are divided into two disjoint classes: the subgroups with
$\mathfrak{so}(2n)^{F}\neq 0$; the subgroups with $\mathfrak{so}(2n)^{F}=0$. If $F$ is a finite abelian subgroup
of $\Spin(2n-2)$ satisfying the condition $(*)$ and with $-1,c_{2n-2}\in F$, then $F\subset\Spin(2n-1)$ is a
subgroup of $\Spin(2n-1)$ in the first class. One can show that, this gives an injective map\footnotemark
from the conjugacy classes of finite abelian subgroups of $\Spin(2n-2)$ satisfying the condition $(*)$ and with
$-1,c_{2n-2}\in F$ to the conjugacy classes of finite abelian subgroups of $\Spin(2n-1)$ in the first class.
\footnotetext{It needs to use a decomposition $\Spin(2n-1)=\Spin(2n-2)B\Spin(2n-2)$ to prove the injectivity
of this map, where $B\cong\Spin(2)$. We leave the proof to interested reader.}
If $F$ is a finite abelian subgroup of $\Spin(2n-1)$ satisfying the condition $(*)$ and with $-1\in F$ and being
in the second class, then $\langle F,c_{2n}\rangle\subset\Spin(2n)$ is a finite abelian subgroup of $\Spin(2n)$
satisfying the condition $(*)$ and with $-1,c_{2n}\in F$. One can show that, this gives a surjective map from
the conjugacy classes of finite abelian subgroups of $\Spin(2n-1)$ in the second class to the conjugacy
classes of finite abelian subgroups of $\Spin(2n)$ satisfying the condition $(*)$ and with $-1,c_{2n}\in F$.
By these, we get the inequality $k_{2n-1}\geq k_{2n-2}+k_{2n}$. Then, by Propositions \ref{P:Spin-3} and
\ref{P:Spin-4}, to show the conclusion we only need to show that $k_{n}>0$ for any even integer $n\geq 16$. An
even integer $n\geq 16$ is of the form $n=4m$ ($m\geq 4$), $n=14+8m$ ($m\geq 1$) or $n=18+8m$ ($m\geq 0$). Let
\begin{eqnarray*}&& F_{4m}=\langle e_1e_3\cdots e_{4m-1},e_1e_2e_{2i+1}e_{2i+2}:1\leq i\leq 2m-1\rangle,
\end{eqnarray*} \begin{eqnarray*}F_{14+8m}&=&\langle e_{8m+1}e_{8m+2}e_{8m+3}e_{8m+4},
e_{8m+1}e_{8m+2}e_{8m+5}e_{8m+6},\\&&e_{8m+1}e_{8m+3}e_{8m+5}e_{8m+7},e_{8m+8}e_{8m+9}e_{8m+10}e_{8m+11},\\&&
e_{8m+8}e_{8m+9}e_{8m+12}e_{8m+13},e_{8m+8}e_{8m+10}e_{8m+12}e_{8m+14},\\&&e_{8i+1}e_{8i+2}e_{8i+3}e_{8i+4},
e_{8i+1}e_{8i+2}e_{8i+5}e_{8i+6},\\&&e_{8i+1}e_{8i+3}e_{8i+5}e_{8i+7},e_{8i+5}e_{8i+6}e_{8i+7}e_{8i+8}:
0\leq i\leq m-1\rangle,\end{eqnarray*}
\begin{eqnarray*}F_{18+8m}&=&\langle e_{8m+1}e_{8m+3}e_{8m+5}e_{8m+7}e_{8m+9}e_{8m+11},e_{8m+1}e_{8m+2}e_{8m+3}e_{8m+4},
\\&&e_{8m+1}e_{8m+2}e_{8m+5}e_{8m+6},e_{8m+1}e_{8m+2}e_{8m+7}e_{8m+8},\\&&e_{8m+1}e_{8m+2}e_{8m+9}e_{8m+10},
e_{8m+12}e_{8m+13}e_{8m+14}e_{8m+15},\\&&e_{8m+12}e_{8m+13}e_{8m+16}e_{8m+17},e_{8m+12}e_{8m+14}e_{8m+16}e_{8m+18},
\\&& e_{8i+1}e_{8i+2}e_{8i+3}e_{8i+4},e_{8i+1}e_{8i+2}e_{8i+5}e_{8i+6},\\&&e_{8i+1}e_{8i+3}e_{8i+5}e_{8i+7},
e_{8i+5}e_{8i+6}e_{8i+7}e_{8i+8}:0\leq i\leq m-1\rangle.\end{eqnarray*}
They are finite abelian subgroups satisfying the condition $(*)$.
\end{proof}

The following proposition is concerned with finite abelian subgroups of the group
$\Spin(12)/\langle c\rangle$, which is needed for the study of abelian subgroups of the simple group
$\E_7/Z(\E_7)$. In $\Spin(12)/\langle c\rangle$, let \begin{eqnarray*}
F_1&=&\langle[e_1e_2e_3e_4e_5e_6],[e_1e_3\frac{1+e_5e_6}{\sqrt{2}}e_7e_9\frac{1+e_{11}e_{12}}{\sqrt{2}}],\\&&
[\frac{e_1+e_2}{\sqrt{2}}\frac{1+e_3e_4}{\sqrt{2}}e_5\frac{e_7+e_8}{\sqrt{2}}\frac{1+e_9e_{10}}{\sqrt{2}}e_{11}]
\rangle,\end{eqnarray*} \[F_2=\langle[-1],[e_1e_2e_3e_4],[e_5e_6e_7e_8],[e_3e_4e_5e_6],[e_7e_8e_9e_{10}],
[e_1e_3e_5e_7e_9e_{11}]\rangle,\] \[F_3=\langle[-1],[e_1e_2e_3e_4],[e_5e_6e_7e_8],[e_3e_4e_5e_6],
[e_7e_8e_9e_{10}],[e_1e_3e_5e_7e_9e_{11}],[\delta]\rangle,\] where
\[\delta=\frac{1+e_1e_2}{\sqrt{2}}\frac{1+e_3e_4}{\sqrt{2}}\frac{1+e_5e_6}{\sqrt{2}}
\frac{1+e_7e_8}{\sqrt{2}}\frac{1+e_9e_{10}}{\sqrt{2}}\frac{1+e_{11}e_{12}}{\sqrt{2}}.\]
The subgroup $F_1$ and $F_3$ are maximal abelian subgroups of $\Spin(12)/\langle c\rangle$ and apparently
$F_2$ is a subgroup of $F_3$.

\begin{prop}\label{P:Half Spin twelve}
Any finite abelian subgroup $F$ of $\Spin(12)/\langle c\rangle$ satisfying the condition $(*)$ and with
$-1\in F$ is conjugate to one of $F_1$, $F_2$, $F_3$.
\end{prop}

\begin{proof}
%Let $F\subset\Spin(12)/\langle c\rangle$ be a finite abelian subgroup satisfying the condition $(*)$
%and with $-1\in F$.
Let $\pi:\Spin(12)/\langle c\rangle\longrightarrow\SO(12)/\langle-I\rangle$ be the natural projection
and $F'=\pi(F)$. As the discussion of abelian subgroups of projective orthogonal groups in \cite{Yu2}
Section 3, one associates a diagonalizable elementary abelian 2-subgroup $B_{F'}$ to $F'$. We consider
separate cases according to $B_{F'}$ contains an element conjugate to $I_{6,6}$ or not.

%Without loss of generality we assume that $B_{F'}$ is contained in the subgroup of diagonal matrices
%in $\SO(12)/\langle-1\rangle$.

If $B_{F'}$ contains an element conjugate to $I_{6,6}$, we may and do assume that
$x_1=[e_1e_2e_3e_4e_5e_6]\in F$ and $\pi(x)\in B_{F'}$. Let $y_1=e_1e_2e_3e_4e_5e_6\in \Spin(12)$. Since
$\pi(x)\in B_{F'}$, one has \[F\subset\Spin(12)^{y_1}/\langle c\rangle\cong(\SU(4)\times
\SU(4))/\langle(iI,iI)\rangle,\] where $y_1=(iI,I)$. By \cite{Yu2}, Proposition 2.1, one can show that
\[F\sim\langle[(iI,I)],[(A_4,A_4)],[(B_4,B_4)]\rangle\] or \[F\sim\langle[(iI,I)],[(I_{2,2},I_{2,2})],
[(J'_2,J'_2)],[(A'_4,A'_4)],[(B'_4,B'_4)]\rangle,\] where \[A'_4=\diag\{I_{1,1},I_{1,1}\},\
B'_4=\diag\{J'_1,J'_1\},\] \[A_4=\frac{1+i}{\sqrt{2}}\diag\{1,i,-1,-i\},\ B_4=
\frac{1+i}{\sqrt{2}}\left(\begin{array}{cccc}0&1&0&0\\0&0&1&0\\0&0&0&1\\1&0&0&0\\\end{array}\right).\]
These two subgroups are conjugate to $F_1$ and $F_2$ respectively.

If $B_{F'}$ contains no elements conjugate to $I_{6,6}$, by \cite{Yu2}, Proposition 3.1, we have integers
$k\geq 0$ and $s_0\geq 1$ such that $12=2^{k}s_0$. Then, $k=0$, $1$ or $2$. In the case of $k=0$, the
subgroup $F'$ is diagonalizable. Hence the pre-image of $F$ in $\Spin(12)$ is abelian. By Proposition
\ref{P:Spin-4}, an abelian subgroup $F$ of $\Spin(12)$ satisfying the condition $(\ast)$ and with $-1,c
\in F$ contains an element conjugate to $e_1e_2e_3e_4e_5e_6$, which is a contradiction. In the case of
$k=1$, one has $s_0=6$. Then, $\log_{2} 6\leq \rank B_{F'}\leq 6-1$. Thus $3\leq \rank B_{F'}\leq 5$. As
$B_{F'}$ contains no elements conjugate to $I_{6,6}$, one can show that $\rank B_{F'}=4$ and
\[B_{F'}\sim\langle I_{4,8},I_{8,4},\diag\{-I_{2,4},I_6\},\diag\{I_6,I_{4,2}\}\rangle.\] Hence
\[F\subset\Spin(12)^{\pi^{-1}(B_{F'})}/\langle c\rangle=(\Spin(2)^{6}/\langle(c_2,\dots,c_2)\rangle)\rtimes
\langle e_1e_3e_5e_7e_9e_{11}\rangle.\] Therefore \[F\sim F_3=\langle[e_1e_2e_3e_4],[e_5e_6e_7e_8],
[e_3e_4e_5e_6],[e_7e_8e_9e_{10}],e_1e_3e_5e_7e_9e_{11},\delta\rangle.\] In the case of $k=2$, one has
$s_0=3$. Then, $\rank B_{F'}=2$ and there exist elements $[x'_1],[x'_2],[x'_3],[x'_4]\in F$,
$x'_1,x'_2,x'_3,x'_4\in\Spin(12)$ with $x'_1x'_2x_1^{'-1}x_2^{'-1}=x'_3x'_4x_3^{'-1}x_4^{'-1}=c$, and
$x'_{i}x'_{j}x_{i}^{'-1}x_{j}^{'-1}=1$ for $i\in\{1,2\},j\in\{3,4\}$. We may assume that
$$B_{F'}=\langle[I_{4,8}],[I_{8,4}]\rangle.$$ Thus $$F\subset(\Spin(12)/\langle c\rangle)^{[e_1e_2e_3e_4],
[e_5e_6e_7e_8]}=\Spin(4)^3/\langle(-1,-1,1),(1,-1,-1),(c,c,c)\rangle,$$ which is isomorphic to
$$\Sp(1)^{6}/\langle(-1,-1,-1,-1,1,1),(1,1,-1,-1,-1,-1),(1,-1,1,-1,1,-1)\rangle.$$ Consideration in
$\Sp(1)^{6}$ shows $x'_1x'_2x_1^{'-1}x_2^{'-1}=x'_3x'_4x_3^{'-1}x_4^{'-1}=c$ and
$x'_{i}x'_{j}x_{i}^{'-1}x_{j}^{'-1}=1$ for $i\in\{1,2\},j\in\{3,4\}$ can not hold.
\end{proof}

Denote by $G=\Spin(12)/\langle c\rangle$. For $F_1$, denote by $x=[e_1e_2e_3e_4e_5e_6]$,
$$x_1=[e_1e_3\frac{1+e_5e_6}{\sqrt{2}}e_7e_9\frac{1+e_{11}e_{12}}{\sqrt{2}}],\ x_2=[\frac{e_1+e_2}{\sqrt{2}}
\frac{1+e_3e_4}{\sqrt{2}}e_5\frac{e_7+e_8}{\sqrt{2}}\frac{1+e_9e_{10}}{\sqrt{2}}e_{11}].$$ Let
$$y=\frac{1+e_{1}e_{7}}{\sqrt{2}}\frac{1+e_{2}e_{8}}{\sqrt{2}}\frac{1+e_{3}e_{9}}{\sqrt{2}}
\frac{1+e_{4}e_{10}}{\sqrt{2}}\frac{1+e_{5}e_{11}}{\sqrt{2}}\frac{1-e_{6}e_{12}}{\sqrt{2}}.$$ Then,
$y\in G^{x}-(G^{x})_0$, $yx_{1}y^{-1}=x_{1}^{-1}$, $yx_{2}y^{-1}=x_{2}$. The action of $y$ on
$F_1/\langle x\rangle\cong(\bbZ/4\bbZ)^2$ is an invertible transformation of determinant $-1\in
(\bbZ/4\bbZ)^{\ast}$. Consideration in $(G^{x})_0$ shows $N_{(G^{x})_0}(F_1)/C_{(G^{x})_0}(F_1)\cong
\Hom((\mathbb{Z}/4\mathbb{Z})^2,\mathbb{Z}/4\mathbb{Z})\rtimes\SL(2,\mathbb{Z}/4\mathbb{Z})$. Hence
$$\Stab_{W(F_1)}(x)\cong\Hom((\mathbb{Z}/4\mathbb{Z})^2,\mathbb{Z}/4\mathbb{Z})\rtimes
\GL(2,\mathbb{Z}/4\mathbb{Z}).$$ On the other hand, there are eight elements in the $W(F_1)$ orbit of $x$
in $F_1$. Hence $|W(F_1)|=8\times 4^2\times|\GL(2,\mathbb{Z}/4\mathbb{Z})|=8\times 4^2\times 16\times 6=
3\times 2^{12}$. Let $K=\{y\in F_1: y^2=1\}$. Since $-1\in Z(G)$, one has
$W(K)\cong\mathbb{F}_{2}^{2}\rtimes\GL(2,\mathbb{F}_2)$. There is a homomorphism $p: W(F_1)\longrightarrow W(K)$.
By counting the order, one can show that $p$ is surjective and $\ker p=\Hom(F_1/K,K)$. Therefore there is
an exact sequence $$1\rightarrow\Hom(\mathbb{F}_{2}^{3},\mathbb{F}_{2}^{3})\rightarrow W(F_1)
\rightarrow \mathbb{F}_{2}^{2}\rtimes\GL(2,\mathbb{F}_2)\rightarrow 1.$$ The subgroup $F_2$ is similar
as the subgroup $F_{12}$ of
$\Spin(12)$, one has $$W(F_{2})=\mathbb{F}_2^{5}\rtimes(\mathbb{F}_{2}^{4}\rtimes S_6).$$
For $F_3$, denote by \[K=\langle[-1],[e_1e_2e_3e_4],[e_5e_6e_7e_8],[e_3e_4e_5e_6],[e_7e_8e_9e_{10}]
\rangle\] and $x=[e_1e_3e_5e_7e_9e_{11}]$. There is a homomorphism $p: W(F_3)\longrightarrow W(K)\times
\GL(F_3/K)$. Apparently $W(K)\cong\mathbb{F}_2^{4}\rtimes S_6$. Since the coset $\delta K$ is stable under
$W(F_3)$, the projection of the image of $p$ to $\GL(F_3/K)$ lies in $S_2=\langle\sigma\rangle$, where
$\sigma(\delta K)=\delta K$, $\sigma(xK)=\delta xK$ and $\sigma(\delta xK)=xK$. An element $w\in\ker p$
is given by the conjugation action of element $g\in C_{G}(K)$. One can show that $C_{G}(K)$ is generated
by $F_3$ together with $\Spin(2)^{6}$, which act as identity on $\delta$. Hence $\ker p\subset
\Hom(F_3/\langle\delta,K\rangle,K)\cong\mathbb{F}_2^{5}$. The conjugation action of elements like $e_1e_2$,
$(\cos\theta+\sin\theta e_{2i-1}e_{2i})(\cos\theta+\sin\theta e_{2i+1}e_{2i+2})$, $1\leq i\leq 5$,
$\theta=\frac{\pi}{4}$ generates $\ker p=\Hom(F_3/\langle\delta,K\rangle,K)$. The conjugation action of
elements like $\frac{1+e_{2i+1}e_{2i+3}}{\sqrt{2}}\frac{1+e_{2i+2}e_{2i+4}}{\sqrt{2}}$, $e_{2i+1}e_{2i+3}$,
$0\leq i\leq 4$ generate $\mathbb{F}_2^{4}\rtimes S_6$. The conjugation action of
$\prod_{1\leq i\leq 6}(\cos\theta+\sin\theta e_{2i-1}e_{2i})$ generates $S_2$, where $\theta=\frac{\pi}{8}$.
Therefore $$W(F_{3})=\mathbb{F}_2^{5}\rtimes((\mathbb{F}_2^{4}\rtimes S_6)\times S_2).$$

\section{Compact simple Lie groups $\G_2$ and $\F_4$}

\subsection{Preliminaries}
We follow the notation in \cite{Yu}, Subsection 3.1. Given a complex simple Lie algebra $\frg$, with a
Cartan subalgebra $\frh$, Killing form denoted by $B$, root system $\Delta=\Delta(\frg,\frh)$. For each
element $\lambda\in\frh^{\ast}$, let $H_{\lambda}\in\frh$ be defined by \[B(H_{\lambda},H)=
\lambda(H),\ \forall H\in\frh,\] in particular we have an element $H_{\alpha}\in\frh$ for each root $\alpha$.
Let \begin{equation} H'_{\alpha}=\frac{2}{\alpha(H_{\alpha})}H_{\alpha},\label{eq:coroot}\end{equation}
which is called a co-root vector. Choose a simple system $\Pi=\{\alpha_1,\alpha_2,\dots,\alpha_{r}\}$ of
the root system $\Delta$, with the Bourbaki numbering (\cite{Bourbaki}, Pages 265-300). For simplicity we
write $H_{i}$, $H'_{i}$ for $H_{\alpha_{i}}$, $H'_{\alpha_{i}}$, $1\leq i\leq r$. For a root $\alpha$,
denote by $X_{\alpha}$ a non-zero vector in the root space $\frg_{\alpha}$. One can normalize the root
vectors $\{X_{\alpha}\}$ appropriately so that the real subspace \begin{equation}\fru_0=
\span_{\bbR}\{X_{\alpha}-X_{-\alpha},i(X_{\alpha}+X_{-\alpha}),i H_{\alpha}:\alpha \in \Delta^{+}\}
\label{eq:compact real form}\end{equation} is a compact real form of $\frg$ (\cite{Knapp}, Pages 348-354).
We denote by $\frf_4$ a compact simple Lie algebra of type $\bf F_4$. Let $\F_4$ be a connected and simply
connected Lie group with Lie algebra $\frf_4$. Let $\frf_4(\bbC)$ and $\F_4(\bbC)$ denote their
complexifications. Similar notations are used for other types. Given a compact semisimple Lie algebra
$\fru_0$, let $\Aut(\fru_0)$ be the group of automorphisms of $\fru_0$ and $\Int(\fru_0)=\Aut(\fru_0)_0$
be the group of inner automorphisms.

\smallskip

Given a connected compact Lie group $G$ and an abelian subgroup $F$, let $\fra_0=\Lie F$,
$\fra=\fra_0\otimes_{\bbR}\bbC$. Denote by $L=C_{G}(\fra_0)$ and $L_s=[L,L]$. Then, $L$ is connected
and $L=Z(L)_0\cdot L_{s}$ (\cite{Knapp}, Page 260, Corollary 4.51).

\begin{lemma}\label{L:center}
If $F\subset G$ is an abelian subgroup satisfying the condition $(*)$, then $\fra=Z(C_{\frg}(\fra))$
and $F\cap L_{s}\subset L_{s}$ also satisfies the condition $(*)$.
\end{lemma}

\begin{proof}
Since $F$ is an abelian subgroup of $G$ satisfying the condition $(*)$, one has $F\subset L$ and
$Z(L)_0\subset F$. Thus $F_0=Z(L)_0$ and $F=F_0(F\cap L_{s})$. Hence $\fra=Z(C_{\frg}(\fra))$ follows.
As $F_0=Z(L)_0$ commutes with $L_{s}$ and $F$ satisfies the condition $(\ast)$, $F\cap L_{s}$ is an
abelian subgroup of $L_{s}$ satisfying the condition $(\ast)$.
\end{proof}

Note that, $C_{\frg}(\fra)$ is a Levi subalgebra. Hence Lemma \ref{L:center} indicates that the
complexified Lie algebra of $F$ is the center of a Levi subalgebra. The following is a well known
property of Levi subgroups.

\begin{lemma}\label{L:Levi}
Let $G$ be a compact connected and simply connected Lie group and $L$ be a closed connected subgroup
of $G$. If $\frl=\Lie L\otimes_{\bbR}\bbC$ is a Levi subalgebra of $\frg$, then $L_{s}=[L,L]$ is
simply connected.
\end{lemma}

\begin{proof}
Choose a maximal torus $T$ of $G$. Let $\frt_0=\Lie T$ and $\frh=\frt_0\otimes_{\bbR}\bbC$. Then,
$\frh$ is a Cartan subalgebra of $\frg$. We have a root system $\Delta$, a simple system $\Pi$ of
$\Delta$, and co-root vectors $H'_{\alpha}=\frac{2}{\alpha(H_{\alpha})}H_{\alpha}\in i\frt_0$. The
simply connectedness of $G$ is equivalent to the kernel of exponential map
$\exp:\frt_0\longrightarrow T$ is equal to $\span_{\bbZ}\{2\pi iH'_{\alpha_{j}}: \alpha_{j}\in\Pi\}$.
Since $\frl$ is a Levi subalgebra, without loss of generality we assume that $T\subset L$ and a subset
$\Pi'$ of $\Pi$ is a simple system of the root system of $L$. Then, $T'=(T\cap L_{s})_0$ is a maximal
torus of $L_s$ and the kernel of exponential map $\exp: \frt'_0\longrightarrow T'$ is equal to
$\span_{\bbZ}\{2\pi iH'_{\alpha_{j}}: \alpha_{j}\in\Pi'\}$. Hence $L_{s}$ is simply connected.
\end{proof}

%One can calculate the coefficients of the above vector $H$ as the determinats of co-factors of the Dynkin
%matrix of $\frg$. It is necessary that these determinants associated to co-factors without the $i$-th row
%are coprime. For example, the $\alpha_2$ root of $A_3$.

\begin{lemma}\label{L:Levi-center}
Let $G$ be a compact connected Lie group and $L$ be a closed connected subgroup of $G$. If $\frl=\Lie L
\otimes_{\bbR}\bbC$ is a Levi subalgebra of $\frg$, then $Z(L_{s})\subset Z(G)Z(L)_0$.
\end{lemma}

\begin{proof}
We only need to consider the case of $G$ is simple and $\rank L_{s}=\rank G-1$, other cases reduce to this
special case. Let $\{\alpha_1,...,\alpha_{r}\}$ be a set of simple roots of $\frg$ with $\{\alpha_1,...,
\alpha_{r}\}-\{\alpha_{i}\}$ being a set of simple roots of $\frl$. Let $H'_{j}=H'_{\alpha_{j}}$ be the
corresponding co-roots (cf. \cite{Huang-Yu}). There exists $(a_1,\dots,a_{r})\in\bbR^{r}-\{0\}$ with
$H=\sum_{1\leq i\leq r}a_{i}H'_{i}$ satisfying that $\alpha_{j}(H)=0$ if and only if $j\neq i$. The group
$Z(L_s)$ consists of elements of the form $x=\exp(2\pi i H')$ for some
$H'=\sum_{1\leq j\leq r,j\neq i}b_{j}H'_{j}$ such that $\alpha_{k}(H')\in\bbZ$ for any $k\neq i$. Given an
$x\in Z(L_s)$, there exists $t\in\bbR$ such that $\alpha_{k}(H'-tH)\in\bbZ$ for any $k$ by the property of
$H$. That just means $x=\exp(2\pi i H')=\exp(2\pi i H'-2\pi i tH)\exp(2\pi i tH)\in Z(G)Z(L)_0$.
\end{proof}

\subsection{Compact simple Lie group of type $\bf G_2$}
Let $G=\G_2$. Denote by $F_1$ an elementary abelian 2-subgroup of $G$ with rank 3. By \cite{Yu},
Corollary 4.2, one has $W(F_1)=N_{G}(F_1)/C_{G}(F_1)\cong\GL(3,\mathbb{F}_2)$.

\begin{prop}\label{P:G2-classification}
Any closed abelian subgroup $F$ of $\G_2$ satisfying the condition $(\ast)$ is either a maximal
torus or is conjugate to $F_1$.
\end{prop}

\begin{proof}
Let $k=\dim F$. Then $0\leq k\leq \rank G=2$. If $k=2$, then $F$ is a maximal torus. If $k=1$,
by Lemma \ref{L:Levi}, $\fra=\Lie F\otimes_{\bbR}\bbC$ is the center of a Levi subalgebra. The
complex simple Lie group $\G_2(\bbC)$ has two conjugacy classes of Levi subgroups with center of
dimension 1, both are isomorphic to $\GL(2,\bbC)$, with maximal compact subgroup isomorphic to
$\U(2)$. However any closed abelian subgroup of $\U(2)$ satisfying the condition $(\ast)$ is a
maximal torus, we get a contraction. If $k=0$, for any $1\neq x\in F$, one has $F\subset G^{x}$
and $Z(G^{x})_0\subset F$. By Steinberg's theorem $G^{x}$ is connected. Hence it is semisimple
as $F$ is a finite subgroup. Therefore $\frg^{x}$ is of type of $A_2^{L}$ or $A_1^{L}+A_1^{S}$.
Thus $G^{x}\cong\SU(3)$ or $(\Sp(1)\times\Sp(1))/\langle(-1,-1)\rangle$. The first case can not
happen since $\SU(3)$ possesses no finite abelian subgroups satisfying the condition $(\ast)$.
In the second case, one can show that $F\sim F_1$.
\end{proof}

\subsection{Compact simple Lie group of type $\bf F_4$ }

Let $G=\F_4$. Denote by $\theta=\exp(\frac{2\pi i}{3}(2H'_3+H'_4))$. Then the root system of
$\frg^{\theta}$ has a simple system $\{\beta-\alpha_1,\alpha_1\}\sqcup\{\alpha_3,\alpha_4\}$,
which is of type $A_2^{L}+A_2^{S}$. Here $\beta=2\alpha_1+3\alpha_2+4\alpha_3+2\alpha_4$ is the
highest root. By Steinberg's theorem $G^{\theta}$ is connected. Since
\[\exp(\frac{2\pi i}{3}(H'_{\beta-\alpha_1}+2H'_{\alpha_1}))=
\exp(\frac{2\pi i}{3}(2H'_{\alpha_3}+H'_{\alpha_4})),\] one has
$G^{\theta}\cong(\SU(3)\times\SU(3))/\langle(\omega I,\omega I)\rangle$,
where $\theta=[(\omega I,I)]$. Let \[F_1=\langle[(\omega I,I)],[(A,A)],[(B,B)]\rangle,\] where
$A=\left(\begin{array}{ccc}1&&\\&\omega&\\&&\omega^{2}\\\end{array}\right)$ and
$B=\left(\begin{array}{ccc}0&1&0\\0&0&1\\1&0&0\\\end{array}\right)$. Then, $F_1$ is an abelian
subgroup of $G$. Moreover
\[C_{G}(F_1)=C_{G^{\theta}}(F_1)=((\SU(3)\times\SU(3))/\langle(\omega I,\omega I)\rangle)^{F_1}=F_1,\]
hence $F_1$ is a maximal abelian subgroup of $G$. For the involution $\sigma_1$ of $G$ (cf.
\cite{Huang-Yu}, Table 2), one has \[G^{\sigma_1}\cong(\Sp(3)\times\Sp(1))/\langle(-I,-1)\rangle.\]
Let \[F_2=\langle[(-I,1)],[(\mathbf{i}I,\mathbf{i})],[(\mathbf{j}I,\mathbf{j})],
[(I_{1,2},1)],[(I_{2,1},1)]\rangle.\] Then, it is an elementary abelian 2-subgroup of $G$. Moreover
it is a maximal abelian subgroup. The group $G$ has a Levi subgroup $L$ of type $\B_3$ and one has
\[L\cong(\Spin(7)\times\U(1))/\langle(-1,-1)\rangle.\] Let
\[F_3=\langle 1\times\U(1),[(e_1e_2e_3e_4,1)],[(e_1e_2e_5e_6,1)],[(e_1e_3e_5e_7,1)]\rangle.\] Since
\[C_{G}(F_3)=C_{L}(F_3)=(1\times\U(1))\cdot(\Spin(7))^{e_1e_2e_3e_4,e_1e_2e_5e_6,e_1e_3e_5e_7}=F_3,\]
$F_3$ is a maximal abelian subgroup of $G$.

\begin{prop}\label{P:F4-classification}
Any closed abelian subgroup $F$ of $G$ satisfying the condition $(*)$ is either a maximal torus, or is
conjugate to one of $F_1,F_2,F_3$.
\end{prop}

\begin{proof}
Denote by $k=\dim F$. Then, $0\leq k\leq\rank G=4$. If $k=4$, then $F$ is a maximal torus. If $0<k<4$,
by Lemma \ref{L:center}, there exists a closed connected subgroup $L$ of $G$ whose complexification is
a Levi subgroup such that $Z(L)_0\subset F\subset L$ and $F'=F\cap L_{s}$ is a finite abelian subgroup
of $L_{s}:=[L,L]$ satisfying the condition $(*)$. By Lemma \ref{L:Levi}, $L_{s}$ is simply connected.
Among compact connected and simply connected Lie groups of rank at most $3$, only $\Spin(7)$ and $\G_2$
has finite subgroups satisfying the condition $(*)$. Hence $L$ is of type $\B_3$. In this case
$L\cong(\Spin(7)\times\U(1))/\langle(-1,-1)\rangle$ and hence $F\sim F_3$ by Proposition \ref{P:Spin-3}.
If $k=0$, for any $1\neq x\in F$, by Steinberg's theorem $G^{x}$ is connected. Hence $G^{x}$ is
semisimple since $F\subset G^{x}$ and $F$ satisfies the condition $(*)$. Thus the root system of
$G^{x}$ is one of the types (cf. \cite{Oshima})
\[A_2^{L}+A_2^{S},\ A_3^{L}+A_1^{S},\ B_4,\ D_4,\ B_2+2A_1,\ 4A_1,\ C_3+A_1,\ C_2+2A_1.\]
If the root system of $G^{x}$ is of type $A_2^{L}+A_2^{S}$, then
\[G^{x}\cong(\SU(3)\times\SU(3))/\langle(\omega I,\omega I)\rangle\] and hence $F\sim F_1$. If the root
system is of type $A_3^{L}+A_1^{S}$, then \[G^{x}\cong(\SU(4)\times\Sp(1))/\langle(-I,-1)\rangle,\]
where $x=[(\pm{i}I,1)]$. Let $S_1,S_2$ be the images of $F$ under the projections to
$\SU(4)/\langle-I\rangle$ and $\Sp(1)/\langle-1\rangle$ respectively, and $m_1$, $m_2$ be
bimultiplicative functions on them. By \cite{Yu2}, Proposition 2.1, one has $\ker m_1=\ker m_2=1$,
$S_1\cong(C_2)^4$ and $S_2\cong(C_2)^2$. Hence $\rank S_1/\ker m_1\neq\rank S_2/\ker m_2$. On the other
hand, for any $[(A,x)],[(B,y)]\in F$, one has $m_1([A],[B])=m_2([x],[y])$. Hence
$\rank(S_1/\ker m_1)=\rank(S_2/\ker m_2)$, which is a contradiction. In the remaining cases, for any
$1\neq x\in F$, the root system of $G^{x}$ is one of of the types
\[B_4,\ D_4,\ B_2+2A_1,\ 4A_1,\ C_3+A_1,\ C_2+2A_1.\] Hence $o(x)=2$. Therefore $F$ is an elementary
abelian 2-group. By \cite{Yu}, Proposition 5.2, one can show that $F\sim F_2$.
\end{proof}

Recall that (cf. \cite{Yu}), $P(r,s;\mathbb{F}_2)$ is the group of $(r,s)$- blockwise upper triangle
matrices in $\GL(r+s,\mathbb{F}_2)$.

\begin{prop}\label{P:Weyl-F4}
One has \[W(F_1)\cong\SL(3,\mathbb{F}_3),\] \[W(F_2)\cong P(3,2;\mathbb{F}_2),\]
\[W(F_3)\cong\mathbb{F}_{2}^{3}\rtimes(\GL(3,\mathbb{F}_2)\times\{\pm{1}\}).\]
\end{prop}

\begin{proof}
The group $W(F_1)$ acts transitively on $F_1-\{1\}$ since any non-identity of $F_1$ is conjugate to
$\theta$. On the other hand the stabilizer of $W(F_1)$ at $\theta$ is isomorphic to
$\mathbb{F}_3^{2}\rtimes\SL(2,\mathbb{F}_3)$. Hence $W(F_1)\cong\SL(3,\mathbb{F}_3)$. By \cite{Yu},
Proposition 5.5, $W(F_2)\cong P(3,2;\mathbb{F}_2)$. For $F_3$, one has $(F_3)_0\cong\U(1)$ and
$F_3/(F_3)_0\cong(\mathbb{F}_2)^3$. Hence there is a homomorphism
$p: W(F_3)\longrightarrow\GL(3,\mathbb{F}_2)\times\{\pm{1}\}$. There is a subgroup $H$ of $G$
isomorphic to $\Spin(9)$ and containing $F_3$. From the inclusion $F_3\subset H$, one can show that
the homomorphism $p$ is surjective and $\ker p\cong\Hom(F_3/(F_3)_0, (F_3)_0)\cong\mathbb{F}_{2}^{3}$.
Therefore $W(F_3)\cong\mathbb{F}_{2}^{3}\rtimes(\GL(3,\mathbb{F}_2)\times\{\pm{1}\})$.
\end{proof}

\section{Compact simple Lie groups of type $\bf D_4$}\label{S:D4}

\subsection{Preliminaries}
Let $G=\Aut(\mathfrak{so}(8))$. Then, $G_0\cong\SO(8)/\langle-I\rangle$ and $G/G_0\cong S_3$. Let $F$ be a
closed abelian subgroup of $G$ satisfying the condition $(*)$. Then $F/(F\cap G_0)\cong 1$, $C_2$ or $C_3$.
Let $G_1=\Spin(8)\rtimes\langle\tau,\theta\rangle$, where $\theta^{2}=1$, $\tau^3=1$, $\theta\tau\theta^{-1}
=\tau^{-1}$, $\theta$ acts on $\mathfrak{so}(8,\bbC)$ as a diagram automorphism fixing the roots
$\alpha_1,\alpha_2$ and permuting the roots $\alpha_3,\alpha_4$, and $\tau$ acts on $\mathfrak{so}(8,\bbC)$
as a diagram automorphism fixing the root $\alpha_2$ and permuting the roots $\alpha_1,\alpha_3,\alpha_4$.
One has $\Spin(8)^{\theta}\cong\Spin(7)$ and $\Spin(8)^{\tau}\cong\G_2$.

\begin{lemma}\label{L:center-outer}
Let $(\fru_0,\eta)$ be the pair $(\mathfrak{so}(8),\tau)$ or $(\mathfrak{so}(8),\theta)$. If $F$ is a closed
abelian subgroup of $G$ satisfying the condition $(*)$ and with $F\cdot G_0=\langle\eta,G_0\rangle$, then
$\fra=\Lie F\otimes_{\bbR}\bbC$ is conjugate to the center of a Levi subalgebra of $\frg^{\eta}$.
\end{lemma}

\begin{proof}
We prove the lemma for the pair $(\mathfrak{so}(8),\tau)$. The proof for the pair $(\mathfrak{so}(8),\theta)$
is similar. Choose a maximal torus $T$ of $(G^{\tau})_0$. By \cite{Yu2}, Lemma 5.4, any element in $\tau G_0$
is conjugate to an element in $\tau T$; and for any element $x\in\tau T$, $T$ is a maximal torus of $(G^{x})_0$.
Substituting $F$ by a subgroup conjugate to it if necessary, we may assume that $F$ contains an element
$x\in\tau T$ and $\fra\subset\frt$. Let $\frl'=\C_{\frg^{\tau}}(\fra)$. We show that $\fra=Z(\frl')$. Since
$\fra\subset\frt\subset\frg^{\tau}$, one has $\fra\subset Z(\frl')$. Since $F$ satisfies the condition $(*)$,
in order to show $Z(\frl')\subset\fra$, we just need to show $\Ad(y)(Z(\frl'))=0$ for any $y\in F$. It suffices
to show $\Ad(C_{G_0}(\fra))(Z(\frl'))=0$ since $F\subset C_{G}(\fra)=\langle\tau\rangle\cdot C_{G_0}(\fra)$ and
$\Ad(\tau)|_{\frl'}=\id$. As $C_{G_0}(\fra)$ is connected, we just need to show $[C_{\frg}(\fra),Z(\frl')]=0$.
Choose a Cartan subalgebra $\frs$ of $\frg$ containing $\frt$. Since $\fra\subset\frs$, one has
\[C_{\frg}(\fra)=\frs\oplus\sum_{\alpha|_{\fra}=0}\bbC X_{\alpha}.\] Thus $[\frs,Z(\frl')]=0$ follows from
$Z(\frl')\subset\frt$ as $\frt$ is a cartan subalgebra of $\frl'$. For any root $\alpha$ with $\alpha|_{\fra}=0$,
if $\tau(\alpha)=\alpha$, then $X_{\alpha}\in\frg^{\tau}$. Hence $X_{\alpha}\in\C_{\frg^{\tau}}(\fra)=\frl'$.
Therefore $[X_{\alpha},Z(\frl')]=0$. If $\tau(\alpha)\neq\alpha$, then $X_{\alpha}+\tau(X_{\alpha})+
\tau^2(X_{\alpha})\in\C_{\frg^{\tau}}(\fra)=\frl'$. Hence $[X_{\alpha}+\tau(X_{\alpha})+\tau^2(X_{\alpha}),
Z(\frl')]=0$. Since $Z(\frl')\subset\frt\subset\frs$ and $X_{\alpha}, \tau(X_{\alpha}),\tau^2(X_{\alpha})$ are
root vectors corresponding to distinct roots, one has $[X_{\alpha},Z(\frl')]=0$. This finishes the proof for
the pair $(\mathfrak{so}(8),\tau)$.
\end{proof}

%Let $\tau$ be an automorphism of $\frg=\mathfrak{so}(8,\bbC)$ defined by
%\begin{eqnarray*} &&\tau(H_{\alpha_1})=H_{\alpha_3},\tau(H_{\alpha_3})=H_{\alpha_1},\\&&
%\tau(H_{\alpha_4})=H_{\alpha_4},\tau\sigma(H_{\alpha_2})=H_{\alpha_2},\\&&
%\tau(X_{\pm{\alpha_1}})=X_{\pm{\alpha_3}},\tau(X_{\pm{\alpha_3}})=X_{\pm{\alpha_1}}, \\&&
%\tau(X_{\pm{\alpha_4}})=X_{\pm{\alpha_4}},\tau(X_{\pm{\alpha_2}})=X_{\pm{\alpha_2}},
%\end{eqnarray*} Then $\tau$ stabilizes $\fru_0$, so it gives an automorphism of $\fru_0$.
%That is, $\tau\in G$. It is clear that $\tau^{3}=1$ and $\frg^{\tau}\cong\frg_2$.

\subsection{Outer automorphisms of order three}
If $F/(F\cap G_0)\cong C_3$, then $\fra=\Lie F\otimes_{\bbR}\bbC$ is conjugate to the center of a Levi
subalgebra of $\mathfrak{so}(8,\bbC)^{\tau}$ by Lemma \ref{L:center-outer}. Let
\[H=(\Sp(1)^{4}/\langle(-1,-1,1,1),(-1,1,-1,1),(-1,1,1,-1)\rangle)\rtimes\langle\tau\rangle\] be a subgroup
of $G$ corresponding to the sub-root system $4A_1=\{\beta=\alpha_1+2\alpha_2+\alpha_3+\alpha_4,\alpha_1,
\alpha_3,\alpha_4\}$ of $\D_4$. One has $\tau(x_1,x_2,x_3,x_4)\tau^{-1}=(x_1,x_4,x_2,x_3)$. Hence the root
system of $(H_0)^{\tau}$ is the sub-root system $A_1^{L}+A_1^{S}$ of the root system of $\frg^{\tau}=\frg_2$.
One has $(G_0)^{\tau}\cong\G_2$, which has an elementary abelian 2-subgroup of rank $3$. Let $F_1$ be generated
by it and $\tau$. The group $G$ has an element $\tau'\in\tau G_0$ of order 3 such that $(G_0)^{\tau'}\cong
\PSU(3)$ (cf. \cite{Helgason}), which has an elementary abelian 3-subgroup of rank $2$. Let $F_2$ be generated
by it and $\tau'$. Let $F_3=\langle\tau, T\rangle$, where $T$ is a maximal torus of $(G^{\tau})_0$. It is clear
that $F_1$, $F_2$, $F_3$ are maximal abelian subgroups of $G$.

\begin{prop}\label{P:D4-5}
Any closed abelian subgroup of $G$ satisfying the condition $(*)$ and with $F/F\cap G_0\cong C_3$ is conjugate
to one of $F_1$, $F_2$, $F_3$.
\end{prop}

\begin{proof}
Let $F$ be a closed abelian subgroup of $G$ satisfying the condition $(*)$ and with $F/F\cap G_0\cong C_3$. If
$F$ is not finite, by Lemma \ref{L:center-outer}, $\fra=\Lie F\otimes_{\bbR}\bbC$ is conjugate to the center of
a Levi subalgebra $\frl'$ of $\frg_2=\mathfrak{so}(8,\bbC)^{\tau}$. Then, the root system of $\frl'$ is of type
$A_1^{L}$, $A_1^{S}$ or $\emptyset$. Let $L=C_{G}(\fra_0)$. Then, $F\subset L$. If the root system of $\frl'$ is
$\emptyset$, then $F\sim F_3$. If the root system of $\frl'$ is of type $A_1^{L}$ or $A_1^{S}$, $L$ is conjugate
to a subgroup of $H$. Precisely \[L\sim L_1=((\Sp(1)\times\U(1)^{3})/\langle(-1,-1,1,1),(-1,1,-1,1),(-1,1,1,-1)
\rangle)\rtimes\langle\tau\rangle\] or \[L\sim L_2=((\U(1)\times\Sp(1)^{3})/\langle(-1,-1,1,1),(-1,1,-1,1),
(-1,1,1,-1)\rangle)\rtimes\langle\tau\rangle.\] For $L_1$, denote by $L''_1$ the subgroup \[\Sp(1)\times
\big((\{(x_1,x_2,x_3)\in\U(1)^{3}: x_1x_2x_3=1\}/\langle(-1,-1,1),(1,-1,-1)\rangle)\rtimes\langle\tau\rangle
\big).\] Then, $Z(L_1)_0\subset F$ and $F''=F\cap L''_1\subset L''_1$ is a finite abelian subgroup satisfying the
condition $(*)$. As $\Sp(1)$ is a direct factor of $L''_1$, it has no finite abelian subgroups satisfying the
condition $(\ast)$. For $L_2$, let \[L''_2=(\Sp(1)^{3}/\langle(-1,-1,1),(1,-1,-1)\rangle)\rtimes\langle\tau\rangle
\big.\] Then, $Z(L''_2)_0\subset F$ and $F''=F\cap L''_2\subset L''_2$ is a finite abelian subgroup satisfying the
condition $(\ast)$. We may assume that $F''$ contains an element of the form $x=(\lambda,1,1)\tau$,
$\lambda\in\U(1)$. Since $F''\subset(L''_2)^{x}$, one has $Z((L''_2)^{x})_0=1$, which forces $\lambda=\pm{1}$.
Hence $F''$ contains an element conjugate to $\tau$. One can show that in this case the condition $(\ast)$ can not
hold. If $F$ is a finite abelian subgroup of $G$, for any $x\in F-F\cap G_0$, by Steinberg's theorem one can show
that $(G_0)^{x}$ is connected. Then, $(G_0)^{x}$ is semisimple since $F$ satisfies the condition $(*)$. As the
root system of $\frg^{x}$ is a sub-root system of the root system of $\frg^{\tau}=\frg_2$, it is one of the types
$\G_2$, $A_2^{S}$, $A_1^{L}+A_1^{S}$, $A_2^{L}$. If it is of type $G_2$, one has $x\sim\tau$. In this case
$F\sim F_1$. If it is of type $A_2^{S}$, one has $x\sim\tau'$. In this case $F\sim F_2$. If it is of type
$A_1^{L}+A_1^{S}$, one has $x^2\sim\tau$. In this case $F\sim F_1$. One can show that the root system of
$\frg^{x}$ can not be of type $A_2^{L}$.
\end{proof}

\begin{prop}\label{P:Weyl-D4-1}
We have $$W(F_1)\cong\GL(3,\mathbb{F}_2)\times\{\pm{1}\},$$ $$W(F_2)\cong(\mathbb{F}_3)^2\rtimes
\GL(2,\mathbb{F}_3),$$ $$W(F_3)\cong\mathbb{F}_3\rtimes(D_6\times\{\pm{1}\}).$$
\end{prop}
\begin{proof}
For $F_1$, both $F_1\cap G_0\cong(\mathbb{F}_2)^{3}$ and $\langle\tau\rangle$ are stable under $W(F_1)$.
Hence $W(F_1)\subset\GL(3,\mathbb{F}_2)\times\{\pm{1}\}$. As $F_1\subset\G_2\times\langle\tau,\theta\rangle$,
one gets $W(F_1)\cong\GL(3,\mathbb{F}_2)\times\{\pm{1}\}$. For $F_2$, since $(G_0)^{\tau'}\cong\PSU(3)$, one
has $\Stab_{W(F_2)}(\tau')\cong\SL(2,\mathbb{F}_3)$. On the other hand, all elements in $F_2-F_2\cap G_0$ are
conjugate to $\tau'$. Hence $|W(F_2)|=18\times |\SL(2,\mathbb{F}_3)|=9\times |\GL(2,\mathbb{F}_3)|$. There is a
homomorphism $p: W(F_2)\longrightarrow\GL(2,\mathbb{F}_3)\times\{\pm{1}\}$ from the action of $W(F_2)$ on
$F_2\cap G_0$ and on $F_2/F_2\cap G_0$. One can show that $\ker p\cong\Hom(F_2/F_2\cap G_0,F_2\cap G_0)\cong
(\mathbb{F}_3)^2$. By comparing the order, we get $|\Im p|=|\GL(2,\mathbb{F}_3)|$. Using one element in the
$\SL_2$ subgroup corresponding to the root $\alpha_2$, one can construct an element $\theta'$ with
$\theta'\tau'\theta'^{-1}=\tau'^{-1}$, normalizing $F_2$, and $\Ad(\theta')|_{(G_0)^{\tau'}}$ is an outer
automorphism of $\PSU(3)$. By this $\Im p=\{(h,t)\in\GL(2,\mathbb{F}_3)\times\{\pm{1}\}:\det h=t\}\cong
\GL(2,\mathbb{F}_3)$ and it is a subgroup of $W(F_2)$. Therfore $W(F_2)\cong(\mathbb{F}_3)^2\rtimes
\GL(2,\mathbb{F}_3)$. For $F_3$, there is a homomorphism $p: W(F_3)\longrightarrow W(T)\times\Aut(F_3/T)$.
One has $W(T)\cong D_6$ and $\Aut(F_3/T)\cong\{\pm{1}\}$. On the other hand, as $F_3\subset\G_2\times\langle
\tau,\theta\rangle$, we have $D_6\times\{\pm{1}\}\subset W(F_3)$, in particular $p$ is surjective. Given
$w\in W(F_3)$, if $w\in\ker p$, then $w=\Ad(g)$ for some $g\in G_0$. Hence $g\in S:=C_{G_0}(T)$, which is
maximal torus of $G_0$. Thus $g\tau g^{-1}\tau^{-1}\in T$. One can show that
$\{(1-\Ad(\tau))x: x\in S\}\cap T\cong\mathbb{F}_3$. Therefore $\ker p\cong\mathbb{F}_3$ and hence
$W(F_3)\cong\mathbb{F}_3\rtimes(D_6\times\{\pm{1}\})$.
\end{proof}

\subsection{Orthogonal group}

The group $\O(8)/\langle-I\rangle$ has three subgroups $H_1$, $H_2$, $H_3$ with
\[H_1\cong(\SO(2)\times\O(6))/\langle(-I,-I)\rangle,\]
\[H_2\cong(\SO(2)\times\SO(2)\times\O(4))/\langle(-I,-I,-I)\rangle,\]
\[H_3\cong(\U(2)\times\O(4))/\langle(-I,-I)\rangle.\]
In $H_1$, let \[F_4=\SO(2)\times\langle I_{3,3},\diag\{I_{1,2},I_{1,2}\},\diag\{I_{2,1},I_{2,1}\}\rangle,\]
\[F_5=\SO(2)\times\langle I_{1,5},I_{3,3},I_{5,1},\diag\{1,-1,1,-1,1,1\}\rangle,\]
\[F_6=\SO(2)\times\langle I_{1,5},I_{2,4},I_{3,3},I_{4,2},I_{5,1}\rangle.\]
In $H_2$, let $$F_7=\SO(2)\times\SO(2)\times\langle I_{1,3},I_{2,2},I_{3,1}\rangle.$$
In $H_3$, let $$F_8=Z(\U(2))\cdot\langle[(1,I_{2,2})],[(iI_{1,1},\diag\{I_{1,1},I_{1,1}\})],
[(J_1,\diag\{J'_1,J_1\})]\rangle.$$ Let $F_9=\langle\theta, T'\rangle$, where $T'$ is a maximal torus
of $(G_0)^{\theta}$.

%In this case the classification up to conjugacy in $G$ and in $\O(8)/\langle-I\rangle$ are the same.

%Let $F\subset G$ be an abelian subgroup satisfying condition $(*)$ and with
%$F/(F\cap G_0)\cong C_2$, then $\fra=\Lie F\otimes_{\bbR}\bbC$ is conjugate to the center
%of a Levi subalgebra of $\mathfrak{so}(7,\bbC)=\mathfrak{so}(8,\bbC)^{\theta}$ by Lemma
%\ref{L:center-outer}. Let $\fra_0=\Lie F$ and $L=C_{G}(\fra_0)$, we may assume that
%$\fra=(\fra_0)_{\bbR}\otimes\bbC$ is the center of a Levi subalgebra $\frl'$ of
%$\frk=\mathfrak{so}(7,\bbC)=\mathfrak{so}(8,\bbC)^{\theta}$.
%When $\dim F=3$, it is clear that $F\sim\langle\theta, T'\rangle$, where
%$T'$ is a maximal torus of $\SO(7)=\PSO(8)^{\theta}$.

\begin{prop}\label{P:D4-3}
Any non-finite closed abelian subgroup $F$ of $G$ satisfying the condition $(\ast)$ and with $F/(F\cap G_0)
\cong C_2$ is conjugate to one of $F_4$, $F_5$, $F_6$, $F_7$, $F_8$, $F_9$.
\end{prop}

\begin{proof}
By Lemma \ref{L:center-outer}, $\fra=\Lie F\otimes_{\bbR}\bbC$ is conjugate to the center of a Levi subalgebra
$\frl'$ of $\mathfrak{so}(7,\bbC)=\mathfrak{so}(8,\bbC)^{\tau}$. Then, the root system of $\frl'$ is of type
$\emptyset$, $A_1^{L}$, $A_1^{S}$, $A_2^{L}$, $A_1^{L}+A_1^{S}$ or $B_2$. Let $L=C_{G}(\fra_0)$. Then,
$F\subset L$. If the root system of $\frl'$ is of type $\emptyset$, then $F\sim F_9$. If the root system of
$\frl'$ is of type $A_1^{S}$, then $L\cong H_2$. In this case $F\sim F_7$. If the root system of $\frl'$ is of
type $A_1^{L}+A_1^{S}$, then $L\cong H_3$. In this case $F\sim F_8$. If the root system of $\frl'$ is of type
$B_2$, then $L\sim H_1$. In this case $F$ is conjugate to one of $F_4$, $F_5$, $F_6$. If the root system of
$\frl'$ is of type $A_1^{L}$, then \[L\cong((\U(2)\times\SO(2)\times\SO(2))/\langle(-I,-I,-I)\rangle)\rtimes
\langle\theta\rangle,\] where $\theta(X,X_1,X_2)\theta^{-1}=(X,X_1,X_2^{-1})$. Let $L''=(\SU(2)\times\SO(2))
\rtimes\langle\theta\rangle$, where $\theta(X,Y)\theta^{-1}=(X,Y^{-1})$. Then, $F'':=F\cap L''$ is a finite
abelian subgroup of $L''$ satisfying the condition $(*)$. Since $\SU(2)$ is a direct factor of $L''$, it has
no finite abelian subgroups satisfying the condition $(\ast)$. If the root system of $\frl'$ is of type
$A_2^{L}$, then \[L\cong((\U(3)\times\SO(2))/\langle(-I,-I)\rangle)\rtimes\langle\theta\rangle,\] where
$\theta(X,Y)\theta^{-1}=(X,Y^{-1})$. Let $L''=(\SU(3)\times\SO(2))\rtimes\langle\theta\rangle$, where
$\theta(X,Y)\theta^{-1}=(X,Y^{-1})$. Then, $F''=F\cap L''$ is a finite abelian subgroup of $L'' $ satisfying
the condition $(*)$. Since $\SU(3)$ is a direct factor of $L''$, it has no finite abelian subgroups satisfying
the condition $(*)$.
\end{proof}

Among $F_4$, $F_5$, $F_6$, $F_7$, $F_8$, $F_9$, the subgroups $F_6$, $F_7$, $F_8$, $F_9$ are maximal abelian
subgroups.

\begin{prop}\label{Weyl-D4-2}
We have $$W(F_6)\cong S_6\times\{\pm{1}\},$$  $$W(F_7)\cong S_4\times D_4,$$   $$|W(F_8)|=2^{6},$$
$$W(F_9)\cong\{\pm{1}\}\rtimes(\{\pm{1}\}^{3}\rtimes S_3).$$
\end{prop}

\begin{proof}
By considering the preservation of eigenvalues, one can show that  $W(F_6)\subset S_6\times\{\pm{1}\}$.
Using $F_6\subset(\O(2)\times\O(6))/\langle(-I,-I)\rangle$, one gets $W(F_6)=S_6\times\{\pm{1}\}$. Similarly
one shows that $W(F_7)\cong S_4\times D_4$. One has $(F_8)_0\cong\U(1)$ being the maximal torus of an $\SU(2)$
subgroup and $C_{G}((F_8)_0)=H_3$. Considering in $H_3$ shows $|W(F_8)|=2^{6}$. For $F_9$, there is a
homomorphism $p: W(F_9)\longrightarrow W(T')$, which is apparently a surjective map. One has $W(T_1)\cong
\{\pm{1}\}^{3}\rtimes S_3$. Using $C_{G}(T)=S\rtimes\langle\theta\rangle$, where $S$ is a maximal torus of
$G_0$, one shows $\ker p\cong\{\pm{1}\}$. Considering $F_9$ as a subgroup of $\O(7)$, one can show that
$W(F_9)\cong\{\pm{1}\}\rtimes(\{\pm{1}\}^{3}\rtimes S_3)$.
\end{proof}

The group $G_0\cong\PSO(8)$ has a subgroup $H_4$ with \[H_4\cong\Sp(1)^{4}/\langle(-1,-1,1,1),(-1,1,-1,1),
(-1,1,1,-1)\rangle.\] In $H_4$, let \[F_{10}=((\U(1)\times\U(1))\cdot\langle[(\textbf{i},\textbf{i})],
[(\textbf{j},\textbf{j})]\rangle)/Z,\]  \[F_{11}=(\U(1)\cdot\langle[(\textbf{i},\textbf{i},\textbf{i})],
[(\textbf{j},\textbf{j},1)],[(1,\textbf{j},\textbf{j})]\rangle)/Z,\] where $Z=\langle(-1,-1,1,1),(-1,1,-1,1),
(-1,1,1,-1)\rangle$. In $H_1$, let \[F_{12}=\SO(2)\times\langle I_{2,4},I_{4,2},\diag\{I_{1,1},I_{1,1},I_2\},
\diag\{I_{2},I_{1,1},I_{1,1}\}\rangle.\] One sees that $F_{10}$ is conjugate to a subgroup of $F_7$, $F_{12}$
is conjugate to a subgroup of $F_6$, and $F_{11}$ is a maximal abelian subgroup of $G$.

%$$H_5\cong\U(4)/\langle-I\rangle.$$

\begin{prop}\label{P:D4-1}
Any non-finite closed abelian subgroup $F$ of $G$ satisfying the condition $(\ast)$ and with $F\subset G_0$
is either a maximal torus, or is conjugate to one of $F_{10}$, $F_{11}$, $F_{12}$.
\end{prop}

\begin{proof}
Let $\fra=\Lie F\otimes_{\bbR}\bbC$, $\frl=C_{\frg}(\fra)$ and $L=C_{G_0}(\fra)$. By Lemma \ref{L:center},
$\fra=Z(\frl)$ and $F':=F\cap L_{s}$ is a finite abelian subgroup of $L_{s} $ satisfying the condition
$(\ast)$, where $L_{s}=[L,L]$. The root system of $\frl$ is of type $\emptyset$, $A_1$, $A_2$, $2A_1$, $A_3$
or $3A_1$. If the root system of $\frl$ is $\emptyset$, then $F$ is a maximal torus. If the root system of
$\frl$ is of type $A_1$, $2A_1$ or $3A_1$, we may assume that $F\subset H_4$. In this case $F$ is conjugate
to $F_{10}$ or $F_{11}$. If the root system of $\frl$ is of type $A_2$ or $A_3$, we may assume that
$F\subset H_1$. In this case $F$ is conjugate to $F_{12}$.
\end{proof}

\begin{prop}\label{Weyl-D4-3}
We have $|W(F_{10})|=3\times 2^{6}$, $|W(F_{11})|=3\times 2^{7}$, and $W(F_{12})\cong S_6\times\{\pm{1}\}$
\end{prop}

\begin{proof}
There is a homomorphism $p: W(F_{10})\longrightarrow W((F_{10})_0)$. Using $F_{10}\subset H_4$ and
an outer involution normalizing $H_4$, one shows that $|\ker p|=6\times 4$. One can show that
$|W((F_{10})_0)|=8$ and $p$ is surjective. Hence $|W(F_{10})|=3\times 2^{6}$. For $F_{11}$, denote
by $x=[(\textbf{i},\textbf{i},\textbf{i})]$, $x_1=[(\textbf{j},\textbf{j},1)]$,
$x_2=[(1,\textbf{j},\textbf{j})]$, There is a homomorphism $p: W(F_{11})\longrightarrow W((F_{11})_0)$,
which is apparently surjective and $W((F_{11})_0)\cong\{\pm{1}\}$. One has $C_{G}((F_{11})_0)\subset H$,
the subgroup defined ahead of Proposition \ref{P:D4-5}. There are eight elements in the $\ker p$ orbit
of $x$ and one has $$C_{G}(\langle(F_{11})_0,x\rangle)=F_{11}\cdot(\U(1)^{3}/Z'\times\langle x_1,x_2\rangle)
\rtimes\langle\theta,\tau\rangle.$$ Hence $|\ker p|=8\times 4\times 6$. Therefore
$|W(F_{11})|=3\times 2^{7}$. Similarly as $F_6$, one can show that $W(F_{12})\cong S_6\times\{\pm{1}\}$.
\end{proof}

\subsection{Finite abelian subgroups}

Recall that $\SO(8)/\langle-I\rangle$ possesses several conjugacy classes of involutions with
representatives $[I_{4,4}]$, $[I_{2,6}]$, $[J_4]$, $[I_{1,7}J_{4}I_{1,7}^{-1}]$. The involution
$[I_{4,4}]$ is not conjugate to any of $[I_{2,6}]$, $[J_4]$, $[I_{1,7}J_4I_{1,7}^{-1}]$ even in
$G=\Aut(\mathfrak{so}(8))$ since the fixed point subalgebras in $\mathfrak{so}(8)$ are non-isomorphic.
However, $[I_{2,6}]$, $[J_4]$, $[I_{1,7}J_{4}I_{1,7}^{-1}]$ are actually conjugate in $G$. This could
be seen in the following way. The elements $x_1=e_1e_2$, $x_2=\Pi$, $x_3=e_1\Pi e_1^{-1}$ of $\Spin(8)$
map to them under the projection $\pi: \Spin(8)\longrightarrow\PSO(8)$, where
$$\Pi=\frac{1+e_1e_2}{\sqrt{2}}\frac{1+e_3e_4}{\sqrt{2}}\frac{1+e_5e_6}{\sqrt{2}}\frac{1+e_7e_8}{\sqrt{2}}.$$
In $\Spin(8)$, $x_1$ and $cx_1$ represent the conjugacy classes of elements $x$ with $x^2=-1$, similarly
$x_2$ and $-x_2$ ($x_3$ and $-x_3$) represent the conjugacy classes of elements $x$ with $x^2=c$ ($x^2=-c$).
As $-1$, $c$, $-c$ are conjugate to each other in $\Spin(8)$, therefore $\pi(x_1)$, $\pi(x_2)$, $\pi(x_3)$
are conjugate to each other in $G$. Note that the centralizer of $[J_4]$ in $\O(8)/\langle-I\rangle$ is
contained in $\SO(8)/\langle-I\rangle$, but the centralizer of $[I_{2,6}]$ in $\O(8)/\langle-I\rangle$
intersects the non-neutral connected component as well. This is a way to distinguish the conjugacy class
of $[I_{2,6}]$ from the conjugacy class of $[J_4]$ in $\O(8)/\langle-I\rangle$.

In $\O(8)/\langle-I\rangle$, let
\begin{eqnarray*}F_{13}&=&\langle[I_{4,4}],[J'_4],[\diag\{I_{2,2},I_{2,2}\}],[\diag\{J'_{2},J'_{2}\}],
[\diag\{I_{1,1},I_{1,1},I_{1,1},I_{1,1}\}],\\&&[\diag\{J'_{1},J'_{1},J'_{1},J'_{1}\}]\rangle,\end{eqnarray*}
\begin{eqnarray*}F_{14}&=&\langle[I_{4,4}],[\diag\{I_{2,2},I_{2,2}\}],[\diag\{J'_{2},J'_{2}\}],
[\diag\{I_{1,1},I_{1,1},I_{1,1},I_{1,1}\}],\\&&[\diag\{J'_{1},J'_{1},J'_{1},J'_{1}\}]\rangle,\end{eqnarray*}
\begin{eqnarray*}F_{15}&=&\langle[I_{4,4}],[\diag\{I_{2,2},I_{2,2}\}],[\diag\{J'_{2},J'_{2}\}],
[\diag\{I_{1,1},I_{1,1},I_{1,1},I_{1,1}\}],\\&&[\diag\{J'_{1},J'_{1},J_{1},J_{1}\}]\rangle,\end{eqnarray*}
\[F_{16}=\langle[I_{4,4}],[\diag\{I_{2,2},I_{2,2}\}],[\diag\{I_{1,1},I_{1,1},I_{1,1},I_{1,1}\}],
[\diag\{J'_{1},J'_{1},J'_{1},J'_{1}\}]\rangle,\]
\[F_{17}=\langle[I_{4,4}],[I_{2,6}],[I_{6,2}],
[\diag\{I_{1,1},I_{1,1},I_{1,1},I_{1,1}\}],[\diag\{J'_{1},J'_{1},J'_{1},J'_{1}\}]\rangle,\]
%\[F_{18}=\langle[I_{4,4}],[\diag\{I_{2,2},I_{2,2}\}],[\diag\{I_{1,1},I_{1,1},I_{1,1},I_{1,1}\}],
%[\diag\{J'_{1},J'_{1},J'_{1},J'_{1}\}]\rangle,\]
\[F_{18}=\langle[I_{4,4}],[I_{2,6}],[I_{6,2}],
[\diag\{I_{1,1},I_{1,1},I_{1,1},I_{1,1}\}],[\diag\{J'_{1},J'_{1},J'_{1},J_{1}\}]\rangle,\]
%\[F_{19}=\langle[I_{4,4}],[\diag\{I_{2,2},I_{2,2}\}],[\diag\{I_{1,1},I_{1,1},I_{1,1},I_{1,1}\}],
%[\diag\{J'_{1},J'_{1},J'_{1},J'_{1}\}]\rangle,\]
\[F_{19}=\langle[I_{4,4}],[I_{2,6}],[I_{6,2}],
[\diag\{I_{1,1},I_{1,1},I_{1,1},I_{1,1}\}],[\diag\{J'_{1},J'_{1},J_{1},J_{1}\}]\rangle,\]
%\[F_{22}=\langle[I_{4,4}],[\diag\{I_{2,2},I_{2,2}\}],[\diag\{I_{1,1},I_{1,1},I_{1,1},I_{1,1}\}],
%[\diag\{J'_{1},J'_{1},J'_{1},J'_{1}\}]\rangle,\]
\[F_{20}=\langle[I_{4,4}],[I_{2,6}],[I_{6,2}],
[\diag\{I_{1,1},I_{1,1},I_{1,1},J_{1}\}],[\diag\{J'_{1},J'_{1},J_{1},I_{1,1}\}]\rangle,\]
\[F_{21}=\langle[I_{4,4}],[\diag\{I_{2,2},I_{2,2}\}],[\diag\{I_{1,1},I_{1,1},I_{1,1},I_{1,1}\}]\}]
\rangle.\]

\begin{prop}\label{P:D4-2}
A non-diagonalizable finite abelian subgroup $F$ of $\O(8)/\langle-I\rangle$ satisfying the condition
$(\ast)$ is conjugate to one of $F_{13}$, $F_{14}$, $F_{15}$, $F_{16}$, $F_{17}$, $F_{18}$, $F_{19}$,
$F_{20}$, $F_{21}$.
\end{prop}

\begin{proof}
By \cite{Yu2}, Section 2, we have a diagonalizable subgroup $B_{F}$ such that $F/B_{F}$ is an elementary
abelian 2-group, an integer $k=\frac{1}{2}\rank(F/ B_{F})$ and another integer $s_0$ such that
$8=s_0\cdot 2^{k}$; moreover the conjugacy class of $F$ as a subgroup of $\O(8)/\langle-I\rangle$ is
determined by the linear structure $(F/B_{F},m,\mu_1,\dots,\mu_{s_0})$, where each $\mu_{i}$ satisfies
$(F/B_{F},m,\mu_{i})\cong V_{0,k;0,0}$ (cf. \cite{Yu}, Subsection 2.4). Since we assume that $F$ is
non-diagonalizable, $k=1$, $2$ or $3$. If $k=3$, then $F\sim F_{13}$. If $k=2$, then $F\sim F_{14}$ or
$F_{15}$. If $k=1$, then $s_0=4$. We have several cases to consider: $\mu_1=\mu_2=\mu_3=\mu_4$;
$\mu_1=\mu_2=\mu_3\neq\mu_4$; $\mu_1=\mu_2\neq\mu_3=\mu_4$; $\mu_1=\mu_2\neq\mu_3,\mu_4$ and $\mu_3
\neq\mu_4$. In the case of $\mu_1=\mu_2=\mu_3=\mu_4$, $F\sim F_{16}$ or $F_{17}$. In the case of
$\mu_1=\mu_2=\mu_3\neq\mu_4$, one has $F\sim F_{18}$. In the case of $\mu_1=\mu_2\neq\mu_3=\mu_4$, one
has $F\sim F_{19}$. In the case of $\mu_1=\mu_2\neq\mu_3,\mu_4$ and $\mu_3\neq\mu_4$, one
has $F\sim F_{20}$. If $k=0$, then $F\sim F_{21}$.
\end{proof}

Note that, $[I_{2,6}]$ and $[J_{4}]$ are conjugate in $G$ and there are no more distinct conjugacy classes
of involutions in $\O(8)/\langle-I\rangle$ which lie in the same conjugacy class in $G$.

\begin{prop}\label{P:D4-6}
In $G$, the subgroups $F_{13}$, $F_{14}$, $F_{16}$, $F_{17}$, $F_{18}$, $F_{19}$, $F_{20}$, $F_{21}$ are
non-conjugate to each other. One has $F_{15}\sim F_{19}$ and the normalizer of $F_{17}$ meets all connected
components of $G$. In $G$, any diagonalizable finite abelian subgroup $F$ of $\SO(8)/\langle-I\rangle$
satisfying the condition$(\ast)$ is conjugate to one of $F_{13}$, $F_{14}$, $F_{16}$, $F_{21}$.
\end{prop}

\begin{proof}
By comparing $|F_{i}|$ and $|\{x\in F_{i}: x^2=1\}|$, one sees that the only possible pairs of conjugate
subgroups among $F_{13}$, $F_{14}$, $F_{15}$, $F_{16}$, $F_{17}$, $F_{18}$, $F_{19}$, $F_{21}$ are
$(F_{14},F_{17})$ and $(F_{15},F_{18})$. For the pair $(F_{14},F_{17})$, we note that $F_{14}$ contains
a Klein four subgroup with involutions all conjugate to $[J_{4}]$ however $F_{17}$ does not contain such
a Klein four subgorup. Hence they are not conjugate in $G$. For the pair $(F_{15},F_{18})$, we note that
$\{x^2: x\in F_{15}\}=\langle[I_{4,4}]\rangle$ and $\{x^2: x\in F_{18}\}=\langle[I_{2,6}]\rangle$.
As $[I_{4,4}]$ and $[I_{2,6}]$ are not conjugate in $G$, one gets that $F_{15}$ and $F_{18})$ are not
conjugate in $G$.

Let $$K=\langle[I_{2,6}],[I_{4,4}],[I_{6,2}],[\diag\{I_{1,1},I_{1,1},I_4\}],[\diag\{I_{1,1},I_4,I_{1,1}\}],
[\diag\{I_4,I_{1,1},I_{1,1}\}]\rangle.$$ Since $\tau[I_{2,6}]\tau^{-1}$ and $[J_4]$ are conjugate in
$\SO(8)/\langle-I\rangle$, $\Ad(\tau) K$ is an elementary abelian 2-subgroup of $\O(8)/\langle-I\rangle$
satisfying the condition $(\ast)$ and with an element conjugate to $[J_4]$, hence being non-diagonalizable.
Since $\rank K=6$, one sees that $\Ad(\tau) K$ is conjugate to $F_{13}$ in $\O(8)/\langle-I\rangle$. Hence
$K$ is conjugate to $F_{13}$ in $G$. Similar argument shows that any diagonalizable finite abelian subgroup
$F$ of $\SO(8)/\langle-I\rangle$ satisfying the condition $(\ast)$ and with an element conjugate to $I_{2,6}$
is conjugate to one of $F_{13}$, $F_{14}$, $F_{16}$. If $F$ is a diagonalizable finite abelian subgroup $F$
of $\SO(8)/\langle-I\rangle$ satisfying the condition $(\ast)$ and without any element conjugate to $I_{2,6}$,
then any non-trivial element of $F$ is conjugate to $I_{4,4}$. Therefore $F$ is conjugate to one of $F_{21}$.

Now $F_{19}$ is a finite abelian subgroup of $\SO(8)/\langle-I\rangle$ satisfying the condition $(\ast)$ and
with an element conjugate to $[I_{2,6}]$ in $\O(8)/\langle-I\rangle$. Hence $\Ad(\tau) F_{19}$ is a finite
abelian subgroup of $\O(8)/\langle-I\rangle$ satisfying the condition $(\ast)$ and with an element conjugate
to $[J_{4}]$ in $\O(8)/\langle-I\rangle$. By comparing the order and the exponent, one sees that
$\Ad(\tau) F_{19}$ is conjugate to $F_{15}$ in $\O(8)/\langle-I\rangle$. Therefore $F_{15}$ is conjugate to
$F_{19}$ in $G$. Similarly, $\Ad(\tau) F_{17}$ is a finite abelian subgroup of $\O(8)/\langle-I\rangle$
satisfying the condition $(\ast)$ and with elements conjugate to $[J_{4}]$ and $[I_{2,6}]$ in
$\O(8)/\langle-I\rangle$. Hence $\Ad(\tau) F_{17}$ is conjugate to $F_{17}$ in $\O(8)/\langle-I\rangle$.
Therefore the normalizer of $F_{17}$ in $G$ meets all connected components of $G$.
\end{proof}

We do not discuss in details the classification of diagonalizable finite abelian subgroups of
$\O(8)/\langle-I\rangle$ satisfying the condition $(\ast)$ and being not contained in $\SO(8)/\langle-I
\rangle$, which requires some combinatorial counting. We only remark that there is a unique conjugacy
class of maximal abelian subgroups among them. Denote by $F_{22}$ the subgroup of diagonal matrices in
$\O(8)/\langle-I\rangle$. Then, $F_{22}$ lies in this conjugacy class. Among $F_{13}$, $F_{14}$, $F_{16}$,
$F_{17}$, $F_{18}$, $F_{19}$, $F_{20}$, $F_{21}$, $F_{22}$, the subgroups $F_{17}$, $F_{18}$, $F_{19}$,
$F_{20}$, $F_{22}$ are maximal abelian subgroups of $G$.

\begin{prop}\label{Weyl-D4-4}
One has $|W(F_{17})|=3^{2}\times 2^{10}$, $|W(F_{18})|=3\times 2^{7}$, $|W(F_{19})|=2^{9}$,
$|W(F_{20})|=2^{8}$, and $W(F_{13})\cong W(F_{22})\cong S_8$.
\end{prop}
\begin{proof}
First note that the normalizers of $F_{18}$, $F_{19}$, $F_{20}$ are contained in $\O(8)/\langle-I\rangle$
since they contain elements conjugate to $[I_{2,6}]$ and contain no elements conjugate to $[J_4]$ in
$\O(8)/\langle-I\rangle$. Let $K=\langle[I_{4,4}],[I_{2,6}],[I_{6,2}]\rangle.$ Then, $K$ is a subgroup
of each of $F_{18}$, $F_{19}$, $F_{20}$ and is stable under the conjugation action of Weyl groups.
Apparently one has $W(K)\cong S_4$. By identifying the images and kernels of the homomorphism
$p: W(F_{i})\longrightarrow W(K)$ for $i=18$, $19$, $20$, we get that $|W(F_{18})|=3\times 2^{7}$,
$|W(F_{19})|=2^{9}$, $|W(F_{20})|=2^{7}$. We could also appeal to \cite{Yu2}, Proposition 3.4 to get
the orders of Weyl groups. By a similar argument one can show that $|W(F_{17})|=3^{2}\times 2^{10}$ by
noting that there are extra elements in the Weyl group coming from normalizer outside
$\O(8)/\langle-I\rangle$ by Proposition \ref{P:D4-6}. The subgroup $F_{13}$ is conjugate to a subgroup
$K'$ of $F_{22}$, which is stable under $W(F_{22})$. Moreover $C_{G}(K')=F_{22}$, hence
$W(F_{13})\cong W(K')=W(F_{22})$. It is clear that $W(F_{22})\cong S_8$.
\end{proof}

Jun Yu \\ School of Mathematics, \\ Institute for Advanced Study, \\ Einstein Drive, Fuld Hall, \\
Princeton, NJ 08540, USA \\
email:junyu@math.ias.edu.

\end{document}